\documentclass[final,onefignum,onetabnum]{siamltex1213}

\pdfoutput=1

\input{preamble}

\newcommand{\trans}{{T}}

\newtheorem{example}[theorem]{Example}
\newtheorem{remark}[theorem]{Remark}
\renewcommand{\tensor}[1]{\cel{\mat{#1}}}
\renewcommand{\elm}[1]{\cel{#1}}
\newcommand{\cmPbar}{\cel{\mat{\bar{P}}}}
\newcommand{\mRbar}{\mat{\bar{R}}} 
\renewcommand{\itn}[1]{_{#1}}
\newcommand{\fullkron}[2]{\underbrace{{#1} \kron \cdots \kron {#1}}_{\text{${#2}$ terms}}}
\newcommand{\bigbetasub}{\beta = 
    \min_{{S\subset \langle n \rangle}} 
		  \Biggl\{ 
			  \underbrace{\min_{{k \in \langle n \rangle}} 
				              \Bigl(\min_{j \in S} \sum_{i \in \bar{S}} { \cP_{ijk}}
											+ \min_{j \in \bar{S}}\sum_{i \in S} { \cP_{ijk} }
									  \Bigr)
				}_{\beta_1}
        + \underbrace{\min_{{j \in \langle n \rangle}}
				                \Bigl( \min_{k \in S} \sum_{i \in \bar{S}} { \cP_{ijk} }
												+ \min_{k \in \bar{S}}\sum_{i \in S}{ \cP_{ijk} }
												\Bigr)
					}_{\beta_2} 
			\Biggr\}}
\newcommand{\bigbeta}{\beta = 
    \min_{{S\subset \langle n \rangle}} 
		  \Biggl\{ 
			  {\min_{{k \in \langle n \rangle}} 
				              \Bigl(\min_{j \in S} \sum_{i \in \bar{S}} { \cP_{ijk}}
											+ \min_{j \in \bar{S}}\sum_{i \in S} { \cP_{ijk} }
									  \Bigr)
				}
        + {\min_{{j \in \langle n \rangle}}
				                \Bigl( \min_{k \in S} \sum_{i \in \bar{S}} { \cP_{ijk} }
												+ \min_{k \in \bar{S}}\sum_{i \in S}{ \cP_{ijk} }
												\Bigr)
					}
			\Biggr\}}			
\usepackage{todonotes}
\title{Multilinear PageRank}

\author{David F. Gleich\thanks{Department of Computer Science, Purdue University (dgleich@purdue.edu, yu163@purdue.edu)}
        \and Lek-Heng Lim\thanks{Computational and Applied Mathematics Initiative, Department of Statistics, University of Chicago (lekheng@uchicago.edu)} \and Yongyang Yu\footnotemark[1]}

\begin{document}

\maketitle

\begin{abstract}
In this paper, we first extend the celebrated PageRank modification to a higher-order Markov chain.  Although this system has attractive theoretical properties, it is computationally intractable for many interesting problems. We next study a computationally tractable approximation to the higher-order PageRank vector that involves a system of polynomial equations called multilinear PageRank. This is motivated by a novel ``spacey random surfer'' model, where the surfer remembers bits and pieces of history and is influenced by this information. The underlying stochastic process is an instance of a vertex-reinforced random walk. We develop convergence theory for a simple fixed-point method, a shifted fixed-point method, and a Newton iteration in a particular parameter regime. In marked contrast to the case of the PageRank vector of a Markov chain where the solution is always unique and easy to compute, there are parameter regimes of multilinear PageRank where solutions are not unique and simple algorithms do not converge. We provide a repository of these non-convergent cases that we encountered through exhaustive enumeration and randomly sampling that we believe is useful for future study of the problem.
\end{abstract}


\pagestyle{myheadings}
\thispagestyle{plain}
\markboth{GLEICH, LIM, AND YU}{MULTILINEAR PAGERANK}

\section{Introduction}

Google devised PageRank to help determine the importance of nodes in a directed graph representing web pages~\cite{page1999-pagerank}. Given a random walk on a directed graph, the PageRank modification builds a new Markov chain that always has a unique stationary distribution. This new random walk models a ``random surfer'' that, with probability $\alpha < 1$ takes a step according to the Markov chain and with probability $1-\alpha$ randomly jumps according to a fixed distribution. If $\mP$ is a \emph{column stochastic} matrix that represents the random walk on the original graph, then the PageRank vector $\vx$ is unique and solves the linear system: 
\[ \vx = \alpha \mP \vx + (1-\alpha) \vv, \]
where $\vv$ is a stochastic vector and $\alpha$ is a probability (Section~\ref{sec:pr} has a formal derivation). The simple Richardson iteration even converges fast for the values of $\alpha$ that are used in practice. 

Although Google described PageRank for the web graph, the same methodology has been deployed in many applications where the importance of nodes provides insight into an underlying phenomena represented by a graph~\cite{morrison2005-generank,freschi2007-proteinrank,Winter-2012-CancerRank,Gleich-preprint-pagerank-beyond}. We find the widespread success of the PageRank methodology intriguing and believe that there are a few important features that contributed to PageRank's success. First and second are the uniqueness and fast convergence. These properties enable reliable and efficient evaluation of the important nodes. Third, in most applications of PageRank, the input graph may contain modeling or sampling errors, and thus, PageRank's jumps are a type of regularization. This may help capture important features in the graph despite the noise.

In this paper, we begin by developing the PageRank modification to a higher-order Markov chain (Section~\ref{sec:hopr}). These higher-order Markov chains model stochastic processes that depend on more history than just the previous state. (We review them formally in Section~\ref{sec:homc}.) In a second-order chain, for instance, the choice of state at the next time step depends on the last two states. However, this structure corresponds to a first-order, or traditional, Markov chain on a tensor-product state-space. We show that higher-order PageRank enjoys the same uniqueness and fast convergence as in the traditional PageRank problem (Theorem~\ref{thm:hopr}); although computing these stationary distributions is prohibitively expensive in terms of memory requirements. 

Recent work by~\citet{Li-2013-tensor-markov-chain} provides an alternative: they consider a rank-1 approximation of these distributions. When we combine the PageRank modification of a higher-order Markov chain with the Li--Ng approximation, we arrive at the \emph{multilinear PageRank problem} (Section~\ref{sec:tensor-pr}).
For the specific case of an $n$-state second-order Markov chain, described by an $n \times n \times n$ transition probability table, the problem becomes finding the solution $\vx$ of the polynomial system of equations: 
\[ \vx = \alpha \mR (\vx \kron \vx) + (1-\alpha) \vv, \]
where $\mR$ is an $n \times n^2$ column stochastic matrix (that represents the probability table), $\alpha$ is a probability, $\kron$ is the Kronecker product, and $\vv$ is a probability distribution over the $n$-states encoded as an $n$-vector.  We have written the equations in this way to emphasize the similarity to the standard PageRank equations. 

One of the key contributions of our work is that the solution $\vx$ has an interpretation as the stationary distribution of a process we describe and call the ``spacey random surfer.'' The spacey random surfer continuously forgets its own immediate history, but does remember the aggregate history and combines the current state with this aggregate history to determine the next state (Section~\ref{sec:process}). This process provides a natural motivation for the multilinear PageRank vector in relationship to the PageRank random surfer. We build on recent advances in vertex reinforced random walks~\cite{Pemantle-1992-vertex-reinforced,Benaim-1997-vrrw} in order to make this relationship precise.


There is no shortage of data analysis methods that involve tensors. These usually go by taking an $m$-way array as an order-$m$ tensor and then performing a tensor decomposition. When $m = 2$, this is often the matrix SVD and the factors obtained give the directions of maximal variation. When $m > 2$, the solution factors lose this interpretation. Understanding the resulting decompositions may be problematic without an identifiability result such as \citet{Anandkumar-preprint-topics}. Our proposal differs in that our tensor represents a probability table for a stochastic process and, instead of a decomposition, we seek an eigenvector that has a natural interpretation as a stationary distribution.   In fact, a general non-negative tensor can be regarded as a contingency table, which can be converted into a multidimensional probability table. These tables may be regarded as the probability distribution of a higher-order Markov chain, just like how a directed graph becomes a random walk. Given the breadth of applications of tensors, our motivation was that the multilinear PageRank vector would be a unique, meaningful stationary distribution that we could compute quickly.

Multilinear PageRank solutions, however, are more complicated . They are not unique for any $\alpha < 1$ as was the case for PageRank, but only when $\alpha < 1/(m-1)$ where $m-1$ is the order of the Markov chain (or $m$ is the order of the underlying tensor) as shown in Theorem~\ref{thm:tpr-unique}. We then consider five algorithms to solve the multilinear PageRank system: a fixed-point method, a shifted fixed-point method, a nonlinear inner-outer iteration, an inverse iteration, and a Newton iteration (Section~\ref{sec:methods}). These algorithms are all fast in the unique regime. Outside that range, we used exhaustive enumeration and random sampling to build a repository of problems that do not converge with our methods. Among the challenging test cases, the inner-outer algorithm and Newton's method has the most reliable convergence properties (Section~\ref{sec:experiments}). Our codes are available for others to use and to reproduce the figures of this manuscript: \url{https://github.com/dgleich/mlpagerank}.

\section{Background} \label{sec:background}
The technical background for our paper includes a brief review of Li and Ng's factorization of the stationary distribution of a higher-order PageRank Markov chain, which we discuss after introducing our notation.

\subsection{Notation}
Matrices are bold, upper-case Roman letter, as in $\mA$; vectors are bold, lower-case Roman letters, as in $\vx$; and tensors are bold, underlined, upper-case Roman letters, as in $\cmP$. We use $\ve$ to be the vector all ones. Individual elements such as $A_{ij}$, $x_i$, or $\cP_{ijk}$ are always written without bold-face. In some of our results, using subscripts is sub-optimal, and we will use Matlab indexing notation instead $A(i,j)$, $x(i)$, or $\cP(i,j,k)$. An order-$m$, $n$-dimensional tensor has $m$ indices that range from $1$ to $n$. We will use $\kron$ to denote the Kronecker product. Throughout the paper, we call a nonnegative matrix $\mA$ column-stochastic if $\sum_i A_{ij} = 1$. A stochastic tensor is a tensor that is nonnegative and where the sum over the first index $i$ is $1$.  We caution our readers that what we call a ``tensor'' in this article really should be called a \textit{hypermatrix}, that is, a specific coordinate representation of a tensor. See \citet{Lim-2014-tensors} for a discussion of the difference between a tensor and its coordinate representation.

We use $S_1, S_2, \dots$ to denote a discrete time stochastic process on the state space $1, \dots, n$. The probability of an event is denoted $\Pr ( S_t = i )$ and $\Pr( S_t = i \mid S_{t-1} = j )$ is the conditional probability of the event. (For those experts in probability, we use this simplifying notation instead of the natural filtration given the history of the process.)

\subsection{PageRank}
\label{sec:pr}
In order to justify our forthcoming use of the term higher-order PageRank, we wish to precisely define a PageRank problem and PageRank vector. The following definition captures the discussions in \citet{langville2006-book}.
\begin{definition}[PageRank] \label{def:pr} 
Let $\mP$ be a column stochastic matrix, let $\alpha$ be a probability smaller than $1$, and let $\vv$ be a stochastic vector. A PageRank vector $\vx$ is the \emph{unique} solution of the linear system: 
\begin{equation} \label{eq:pr}
 \vx = \alpha \mP \vx + (1-\alpha) \vv. 
\end{equation}  We call the set $(\alpha, \mP, \vv)$ a PageRank problem.
\end{definition}

Note that the PageRank vector $\vx$ is equivalently a Perron vector of the matrix: 
\[ \mM = \alpha \mP + (1-\alpha) \vv \ve^\trans \]
under the normalization that $\vx \ge 0$ and $\ve^\trans \vx = 1$.  The matrix $\mM$ is column stochastic and encodes the behavior of the random surfer that, with probability $\alpha$, transitions according to the Markov chain with transition matrix $\mP$, and, with probability $(1-\alpha)$ ``teleports'' according to the fixed distribution $\vv$. When PageRank is used with a graph, then $\mP$ is almost always defined as the random walk transition matrix for that graph. If a graph does not have a valid transition matrix, then there are a few adjustments available to create one~\cite{boldi2007-traps}. 

When we solve for $\vx$ using the power method on the Markov matrix $\mM$ or the Richardson iteration on the linear system~\eqref{eq:pr}, then we iterate: 
\[ \vx\itn{0} = \vv \qquad \vx\itn{k+1} = \alpha \mP \vx\itn{k} + (1-\alpha) \vv. \]
This iteration satisfies the error bound
\[ \normof[1]{\vx\itn{k} - \vx} \le 2 \alpha^k \]
for any stochastic $\vx\itn{0}$. For values of $\alpha$ between $0.5$ and $0.99$, which occur most often in practice, this simple iteration converges quickly.

\subsection{Higher-order Markov chains}
\label{sec:homc}
We wish to extend PageRank to higher-order Markov chains and so we briefly review their properties.
An $m$\textsuperscript{th}-order Markov chain $S$ is a stochastic process that satisfies:
\[
\begin{aligned}
\Pr(S_t=i_1 \mid S_{t-1}=i_2, \dots, S_1 = i_t)
=\Pr(S_t = i_1 \mid S_{t-1}=i_2, \dots, S_{t-m} = i_{m+1}).
\end{aligned}
\]
In words, this means that the future state only depends on the past $m$ states. Although the probability structure of a higher-order Markov chain breaks the fundamental Markov assumption, any higher-order Markov chain can be reduced to a first-order, or standard, Markov chain by taking a Cartesian product of its state space. Consider, for example, a second-order $n$-state Markov chain $S$. Its transition probabilities are $\cP_{ijk}=\Pr(S_{t+1}=i\mid S_t=j, S_{t-1}=k)$. We will represent these probabilities as a tensor $\cmP$. The stationary distribution equation for the resulting first-order Markov chain satisfies
\[
\sum_{k} \cP_{ijk} X_{jk} = X_{ij} ,
\]
where $X_{jk}$ denotes the stationary probability on the product space.  Here, we have induced an $n^2 \times n^2$ eigenvector problem to compute such a stationary distribution. For such first-order Markov chains, Perron-Frobenius theory~\cite{Perron-1907-theorem,Frobenius-1908-theorem,varga1962-book} governs the conditions when the stationary distribution exists. However, in practice for a $100,000 \times 100,000 \times 100,000$ tensor, we need to store $10,000,000,000$ entries in $\mX = [ X_{ij} ]$. This makes it infeasible to work with large, sparse problems. 

\subsection{Li and Ng's approximation}
\label{sec:li}
As a computationally tractable alternative to working with a first-order chain on the product state-space, \citet{Li-2013-tensor-markov-chain} define a new type of stationary distribution for a higher-order Markov chain. Again, we describe it for a second-order chain for simplicity. For each term $X_{ij}$ in the stationary distribution they substitute a product $x_i x_j$, and thus for the matrix $\mX$ they substitute a rank-1 approximation $\mX = \vx \vx^\trans$  where $\sum_i x_i = 1$. Making this substitution and then summing over $j$ yields an eigenvalue expression called an $l^2$-eigenvalue by \citet{Lim-2005-eigenvalues} and called a $Z$-eigenvalue by \citet{Qi-2005-Z-eigenvalues} (one particular type of tensor eigenvalue problem) for $\vx$:
\[
\sum_j\Big(\sum_k \elm{P}_{ijk}x_j x_k \Big) = \sum_j x_i x_j = x_i \quad \Leftrightarrow \quad \tensor{P} \vx^2 = \vx,
\]
where we've used the increasingly common notational convention: 
\[ \textstyle [\tensor{P} \vx^2 ]_i = \sum_{jk} \elm{P}_{ijk} x_j x_k \] from \citet{Qi-2005-Z-eigenvalues}.
All of these results extend beyond second-order chains, in a relatively straightforward manner. Li and Ng present a series of theorems that govern existence and uniqueness for such stationary distributions that we revisit later. 

\section{Higher-order PageRank} \label{sec:hopr}
Recall the PageRank random surfer. With probability $\alpha$, the surfer transitions according to the Markov chain; and with probability $1-\alpha$, the surfer transitions according to the fixed distribution $\vv$. We define a higher-order PageRank by modeling a random surfer on a higher-order chain. With probability $\alpha$, the surfer transitions according to the higher-order chain; and with probability $1-\alpha$, the surfer teleports according to the distribution $\vv$. That is, if $\cmP$ is the transition tensor of the higher-order Markov chain, then the higher-order PageRank chain has a transition tensor $\cmM$ where   
\[ \cM(i,j,\dots,\ell,k) = \alpha \cP(i,j,\dots,\ell,k) + (1-\alpha) v_i. \]
Recall that any higher-order Markov chain can be reduced to a first-order chain by taking a Cartesian product of the state space. We call this the reduced form a higher-order Markov chain and in the following example, we explore the reduced form of a second-order PageRank modification.

\begin{example} \label{ex:simple-3}
Consider the following transition probabilities: 
\[
\cP(\cdot , \cdot ,1)=\begin{bmatrix}
				0 & \tfrac{1}{2} & 0\\
				0 & 0 & 0\\
				1 & \tfrac{1}{2} & 1
			\end{bmatrix};\quad
\cP( \cdot , \cdot ,2)=\begin{bmatrix}
				\frac{1}{2} & 0 & 1 \\
				0 & \frac{1}{2} & 0 \\
				\frac{1}{2} & \frac{1}{2} & 0
			\end{bmatrix};\quad
\cP(\cdot , \cdot ,3)=\begin{bmatrix}
				\frac{1}{2} & \frac{1}{2} & 0 \\
				0 & \frac{1}{2} & 0 \\
				\frac{1}{2} & 0 & 1
			\end{bmatrix}.
\]
Figure~\ref{fig:hyper_tran} shows the state-space transition diagram for the reduced form of the chain before and after its PageRank modification.
\end{example}

We define a higher-order PageRank \emph{tensor} as the stationary distribution of the reduced Markov chain, organized so that $\cX(i,j,\dots,\ell)$ is the stationary probability associated with the sequence of states $\ell \to \cdots \to j \to i$.

\begin{definition}[Higher-order PageRank]
Let $\cmP$ be an order-$m$ transition tensor representing an $(m-1)$\textsuperscript{th} order Markov chain,
$\alpha$ be a probability less than $1$, and $\vv$ be a stochastic vector. Then the higher-order PageRank tensor $\cmX$ is the order-$(m-1)$, $n$-dimensional tensor that solves the linear system: 
\[ \cX(i,j,\dots,\ell) = \alpha \sum_{k} \cP(i,j,\dots,\ell,k) \cX(j, \dots, \ell, k) + (1-\alpha) v_i \sum_{k} \cX(j,\dots,\ell,k). \]
\end{definition}

For the second-order case from Example~\ref{ex:simple-3}, we now write this linear system in a more traditional matrix form in order to make a few observations about its structure. Let $\mX$ be the PageRank tensor (or matrix, in this case). We have:
\begin{equation}
\label{eq:1-step chain}
\text{vec}(\mX) = \left[\alpha \mP + (1-\alpha)\mV \right] \text{vec}(\mX),
\end{equation}
where $\mP, \mV \in \RR^{n^2 \times n^2}$, and $\mV =  \ve^\trans \otimes \eye \otimes \vv$. In this setup, the matrix $\mP$ is sparse and highly structured:
\[ \mP = \bmat{ 
0 & 0   & 0 & 1/2 & 0   & 0 & 1/2 &   0 & 0\\
0 & 0   & 0 & 0   & 0   & 0 & 0   &   0 & 0\\
1 & 0   & 0 & 1/2 & 0   & 0 & 1/2 &   0 & 0\\ 
0 & 1/2 & 0 & 0   & 0   & 0 & 0   & 1/2 & 0\\
0 & 0   & 0 & 0   & 1/2 & 0 & 0   & 1/2 & 0\\
0 & 1/2 & 0 & 0   & 1/2 & 0 & 0   &   0 & 0\\
0 & 0   & 0 & 0   & 0   & 1 & 0   &   0 & 0\\
0 & 0   & 0 & 0   & 0   & 0 & 0   &   0 & 0\\
0 & 0   & 1 & 0   & 0   & 0 & 0   &   0 & 1\\
}. \] 
When $\alpha = 0.85$ and $\vv = (1/3) \ve$, the higher-order PageRank matrix is: 
\[ \mX = \bmat{     0.0411 &   0.0236  &  0.0586 \\
                    0.0062 &   0.0365  &  0.0397 \\
                    0.0761 &   0.0223  &  0.6959 }. \]
More generally, both $\mP$ and $\mV$ have the following structure for the second-order case:
\[
\sbmat{
\elm{P}_{111}, v_1 & 0 & \cdots & 0 & \elm{P}_{112}, v_1 & 0 & \cdots & 0 & \cdots & \elm{P}_{11n}, v_1 & 0 & \cdots & 0\\
\elm{P}_{211}, v_2  & 0 & \cdots & 0 & \elm{P}_{212}, v_2 & 0 & \cdots & 0 & \cdots & \elm{P}_{21n}, v_2 & 0 & \cdots & 0\\
\vdots & \vdots & \ddots & \vdots & \vdots & \vdots & \ddots & \vdots & \ddots & \vdots & \vdots & \ddots &\vdots\\
\elm{P}_{n11}, v_n & 0 & \cdots & 0 & \elm{P}_{n12}, v_n & 0 & \cdots & 0 & \cdots & \elm{P}_{n1n}, v_n & 0 & \cdots & 0\\
\vdots & \vdots &  \ddots & \vdots & \vdots & \vdots & \ddots & \vdots & \ddots & \vdots & \vdots & \ddots &\vdots\\
0 & \cdots & 0 & \elm{P}_{1n1}, v_1 & 0 & \cdots & 0 & \elm{P}_{1n2}, v_1 & \cdots & 0 & \cdots & 0 & \elm{P}_{1nn}, v_1\\ 
0 & \cdots & 0 & \elm{P}_{2n1}, v_2 & 0 & \cdots & 0 & \elm{P}_{2n2}, v_2 & \cdots & 0 & \cdots & 0 & \elm{P}_{2nn}, v_2\\ 
\vdots & \ddots & \vdots & \vdots & \vdots & \ddots & \vdots & \vdots & \ddots & \vdots & \ddots & \vdots &\vdots\\
0 & \cdots & 0 & \elm{P}_{nn1}, v_n & 0 & \cdots & 0 & \elm{P}_{nn2}, v_n & \cdots & 0 & \cdots & 0 & \elm{P}_{nnn}, v_n\\
}. 
\]

\begin{figure}[!ht]
\centering        
\subfigure[The higher-order Markov chain]{\includegraphics[width=0.45\linewidth]{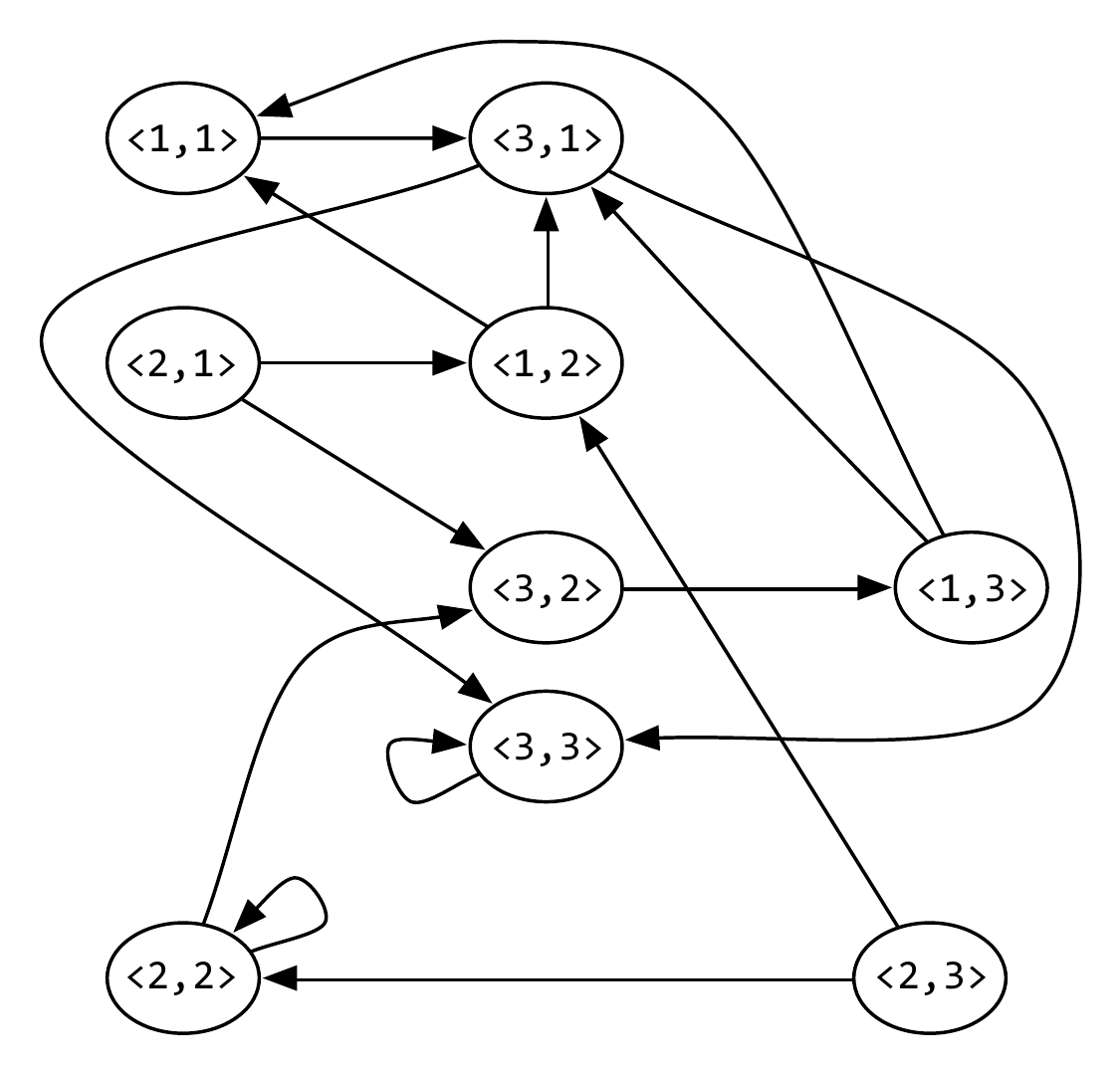}}
\subfigure[The higher-order PageRank chain]{\includegraphics[width=0.45\linewidth]{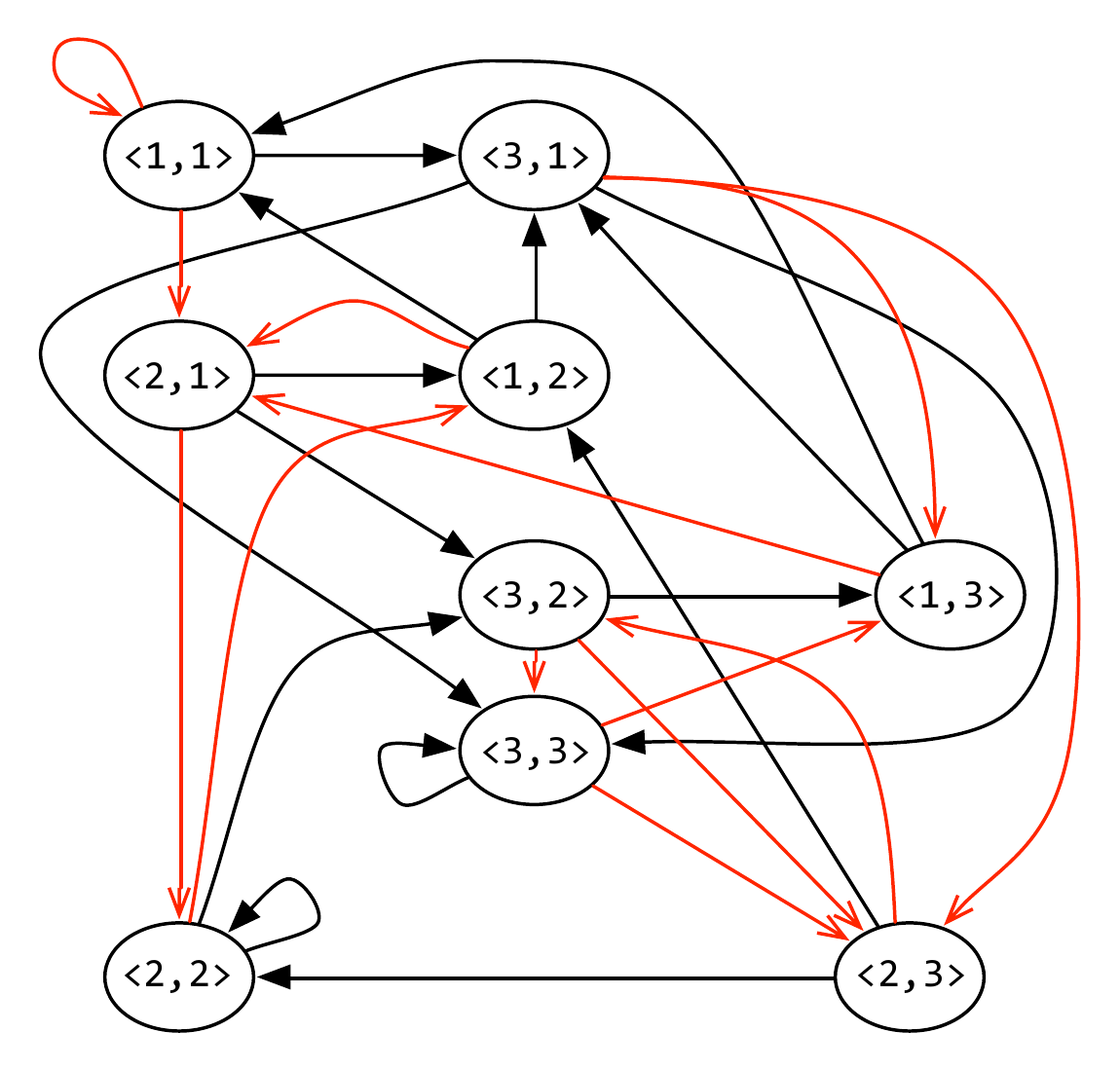}}
\caption{The state space transitions for a higher-order Markov chain on the product-space and the PageRank modification of that same chain with new transitions indicated in red. The transitions for both chains must satisfy $\langle j,k \rangle \to \langle i,j \rangle$. Note that, unlikely the PageRank modification of a first-order Markov chain, the reduced form of the higher-order PageRank chain does not have a complete set of transitions. For instance, there is no transition between $\langle 2,3 \rangle$ and $\langle 1,3 \rangle$. }
\label{fig:hyper_tran}
\end{figure}

In the remainder of this section, we wish to show the relationship between the reduced form of a higher-order PageRank chain and the definition of the PageRank problem (Definition~\ref{def:pr}). This is not as trivial as it may seem! For instance, in the second-order case, equation~\eqref{eq:1-step chain} is not of the correct form for Definition~\ref{def:pr}. But a slight bit of massaging produces the equivalence. 

Consider the vectorized equation for the stationary distribution matrix for the second-order case (from equation~\ref{eq:1-step chain}) as: 
\[ \tvec(\mX) = \underbrace{[\alpha \mP + (1-\alpha) \mV]}_{\mM} \tvec(\mX). \] 
Our goal is to derive a PageRank problem in the sense of Definition~\ref{def:pr} to find $\tvec(\mX)$. As it turns out, $\mM^2$ will give us this PageRank problem. The idea is this: in the first-order PageRank problem we lose all history after a single teleportation step by construction. In this second-order PageRank problem, we keep around one more state of history, hence, two steps of the second-order chain are required to see the effect of teleportation as in the standard PageRank problem. Formally, the matrix $\mM^2$ can be written in terms of matrix $\mP$ and $\mV$, i.e., 
\[
\mM^2 = \alpha^2 \mP^2 + \alpha(1-\alpha)\mP\mV + \alpha(1-\alpha)\mV\mP+(1-\alpha)^2\mV^2.
\]
We now show that $\mV^2 = (\vv \kron \vv)(\ve^\trans \kron \ve^\trans)$ by exploiting two properties of the Kronecker product: $(\mA \kron \mB)(\mC \kron \mD) = (\mA \mC) \kron (\mB \mD)$ and $\va^\trans \kron \vb = \vb \va^\trans$. Note that: 
\[ \mV^2 = (\ve^\trans \kron \mI \kron \vv) (\ve^\trans \kron \mI \kron \vv) = [ \ve^\trans (\ve^\trans \kron \mI) ] \kron [ (\mI \kron \vv) \vv] = (\ve^\trans \kron \ve^\trans) \kron (\vv \kron \vv). \]
This enables us to write a PageRank equation for $\tvec(\mX)$: 
\[ \begin{aligned}
 \tvec(\mX) & = \mM^2 \tvec(\mX) \\
            & = \alpha (2-\alpha) \underbrace{\left[ \tfrac{\alpha}{2-\alpha} \mP^2 + \tfrac{1-\alpha}{2-\alpha} (\mP \mV + \mV \mP) \right] }_{\mP_{\text{pr}}} \tvec(\mX) + (1 - 2 \alpha + \alpha^2) \vv \kron \vv, \end{aligned}\]
where we used the normalization $\ve^\trans \tvec(\mX) = 1$. Thus we conclude:
\begin{lemma} Consider a second-order PageRank problem $\alpha, \cmP, \vv$. Let $\mP$ be the matrix for the reduced form of $\cmP$. Let $\mM = \alpha \mP + (1-\alpha) \mV$ be the transition matrix for the vector representation of the stationary distribution $\mX$. This stationary distribution is the PageRank vector of a PageRank problem $(2\alpha-\alpha^2, \mP_{pr}, \vv\kron \vv)$ in the sense of Definition~\ref{def:pr} with
\[
\mP_{\text{pr}} =\frac{\alpha}{2-\alpha}\mP^2 + \frac{1-\alpha}{2-\alpha}\mP\mV + \frac{1-\alpha}{2-\alpha}\mV\mP.
\]
\end{lemma}

And we generalize: 

\begin{theorem} 
Consider a higher-order PageRank problem $\alpha, \cmP, \vv$ where $\cmP$ is an order-$m$ tensor.  Let $\mP$ be the matrix for the reduced form of $\cmP$. Let $\mM = \alpha \mP + (1-\alpha) \mV$ be the transition matrix for the vector representation of the order-$(m-1)$, $n$-dimensional stationary distribution tensor $\cmX$. This stationary distribution is equal to the PageRank vector of the PageRank problem 
\[ (1-(1-\alpha)^{m-1}, \mP_{\text{pr}}, \fullkron{\vv}{m-1}), \text{ where }
\mP_{\text{pr}} = \frac{\mM^{m-1}-(1-\alpha)^{m-1}\mV^{m-1}}{1-(1-\alpha)^{m-1}}. \]
\end{theorem}
\begin{proof}
We extend the previous proof as follows. The matrix $\mM$ is nonnegative and has only a single recurrent class of all nodes consisting of all nodes in the \emph{reach} of the set of non-zero entries in $v_i$. Thus, the stationary distribution is unique. We need to look at the $m-1$ step transition matrix to find the PageRank problem. Consider $\vec{\cmX}$ as the stationary distribution eigenvector of the $m-1$ step chain: 
\[ \tvec(\cmX) = \mM \tvec(\cmX) = \mM^{m-1} \tvec(\cmX). \]
The matrix $\mM^{m-1}$ can be written in terms of matrix $\mP$ and $\mV$, i.e.,
\[
\mM^{m-1} = \bigl( (\alpha \mP + (1-\alpha)\mV)^{m-1} - (1-\alpha)^{m-1}\mV^{m-1} \bigr) + (1-\alpha)^{m-1}\mV^{m-1}.
\]
The matrix $\mV$ has the structure 
\[ \mV = \ve^\trans \kron (\fullkron{\mI}{m-2})\kron \vv. \]
We now expand $\mV^{m-1}$ using the property the property of Kronecker products $(\mA \kron \mB)(\mC \kron \mD) = (\mA \mC) \kron (\mB \mD)$, repeatedly: 
\[ \begin{aligned}
\mV^{m-1} & = \Bigl[ \ve^\trans \kron (\fullkron{\mI}{m-2})\kron \vv \Bigr] \cdots \Bigl[ \ve^\trans \kron (\fullkron{\mI}{m-2})\kron \vv \Bigr] \\
& = 
	\Bigl[ \ve^\trans(\ve^\trans \kron \mI)(\ve^\trans \kron \mI \kron \mI) \cdots (\ve^\trans \kron \fullkron{\mI}{m-2} ) \Bigr] \kron  \\
	& \qquad \Bigl[ (\fullkron{\mI}{m-2} \vv ) (\mI \kron \mI \kron \vv) (\mI \kron \vv) \vv \Bigr]\\
& = (\fullkron{\ve^\trans}{m-1}) \kron (\fullkron{\vv}{m-1}) \\
 &  = (\fullkron{\vv}{m-1})(\fullkron{\ve^\trans}{m-1}).
\end{aligned}
\]
At this point, we are essentially done as we have shown that the \emph{stochastic} $\mM^{m-1}$ has the form $\mM^{m-1} = \mZ + (1-\alpha)^{m-1} (\vv \kron \cdots \kron \vv)(\ve^\trans \kron \cdots \kron \ve^\trans).$ The statements in the theorem follow from splitting $\mM^{m-1} = \alpha_{\text{pr}} \mP_{\text{pr}} + (1-\alpha_{\text{pr}}) (\vv \kron \cdots \kron \vv)(\ve^\trans \kron \cdots \kron \ve^\trans)$, that is,
\[ \alpha_{\text{pr}} = 1- (1-\alpha)^{m-1} \qquad \mP_{\text{pr}} = \frac{1}{\alpha_{\text{pr}}} \left( \mM^{m-1} - (1-\alpha) \mV^{m-1} \right). \]
The matrix $\mP_{\text{pr}}$ is stochastic because the final term in the expansion of $\mM^{m-1}$ is $(1-\alpha) \mV^{m-1}$, thus, the remainder is a nonnegative matrix with column sums equal to a single constant less than $1$.
\end{proof}

\begin{corollary} \label{thm:hopr}
The higher-order PageRank stationary distribution tensor $\cmX$ always exists and is unique. Also, the standard PageRank iteration will result in a $1$-norm error of $2\bigl(1-(1-\alpha)^{m-1}\bigr)^k$ after $(m-1)k$ iterations.
\end{corollary}

Hence, we retain all of the attractive features of PageRank in the higher-order PageRank problem. The uniqueness and convergence results in this section are not overly surprising and simply clarify the relationship between the higher-order Markov chain and its PageRank modification.

\section{Multilinear PageRank}
\label{sec:tensor-pr}

The tensor product structure in the state-space and the higher-order stationary distribution make the straightforward approaches of the previous section difficult to scale to large problems, such as those encountered in modern bioinformatics and social network analysis applications. The scalability limit is the memory required. Consider an $n$-state, second-order PageRank chain: $(\alpha, \cmP, \vv)$. It requires $O(n^2)$ memory to represent the stationary distribution, which quickly grows infeasible as $n$ scales. To overcome this scalability limitation, we consider the Li and Ng approximation to the stationary distribution with the additional assumption:

\textsc{Assumption}.
\emph{There exists a method to compute $\cmP \vx^2$ that works in time proportional to the memory used for to represent $\cmP$.}

This assumption mirrors the fast matrix-vector product operator assumption in iterative methods for linear systems. Although here it is critical because there must be at least $n^2$ non-zeros in any second-order stochastic tensor $\cmP$. If we could afford that storage then the higher-order techniques from the previous section would apply and we would be under the scalability limit. We discuss how to create such fast operators from sparse datasets in Section~\ref{sec:fast-operators}.

The Li and Ng approximation to the stationary distribution of a second-order Markov chain replaces the stationary distribution with a symmetric rank-1 factorization: $\mX = \vx \vx^\trans$ where $\sum_i x_i = 1$. For a second-order PageRank chain, this transformation yields an implicit expression for $\vx$: 
\begin{equation} \label{eq:tensorpr1}
\vx = \alpha \cmP \vx^2 + (1-\alpha) \vv. 
\end{equation} 
We prefer to write this equation in terms of the Kronecker product structure of the tensor flattening along the first index. Let $\mR := \cmP_{(1)} = \flat_1(\cmP)$ be the $n$-by-$n^2$, stochastic flattening (see \citealt[Section 12.4.5]{Golub-2013-book} for more on \emph{flattenings} or \emph{unfoldings} of a tensor, and \citealt{Draisma-2014-bounded-rank} for the $\flat$ notation) along the first index: 
\[
\mR = \left[
\begin{array}{c c c | c c c | c | c c c}
\elm{P}_{111} & \cdots & \elm{P}_{1n1} & \elm{P}_{112} & \cdots & \elm{P}_{1n2}  & \cdots & \elm{P}_{11n} & \cdots & \elm{P}_{1nn}\\
\elm{P}_{211} & \cdots & \elm{P}_{2n1} & \elm{P}_{212} & \cdots & \elm{P}_{2n2} & \cdots & \elm{P}_{21n} & \cdots & \elm{P}_{2nn}\\
\vdots & \ddots & \vdots &  \vdots & \ddots & \vdots & \vdots & \vdots & \ddots & \vdots \\
\elm{P}_{n11} & \cdots & \elm{P}_{nn1}  & \elm{P}_{n12} & \cdots & \elm{P}_{nn2} & \cdots & \elm{P}_{n1n} & \cdots & \elm{P}_{nnn}\\
\end{array} \right].
\]
Then equation~\ref{eq:tensorpr1} is: 
\[ \vx = \alpha \mR (\vx \kron \vx) + (1-\alpha) \vv. \]
Consider the tensor $\cmP$ from Example~\ref{ex:simple-3}. The multilinear PageRank vector for this case with $\alpha = 0.85$ and $\vv = (1/3) \ve$ is: 
\[ \vx = \bmat{     0.1934 \\  0.0761 \\  0.7305 } \]
We generalize this second-order case to the order-$m$ case in the following definition of the multilinear PageRank problem. 

\begin{definition}[Multilinear PageRank]
 Let $\cmP$ be an order-$m$ tensor representing an $(m-1)$\textsuperscript{th} order Markov chain, $\alpha$ be a probability less than $1$, and $\vv$ be a stochastic vector. Then the multilinear PageRank vector is a nonnegative, stochastic solution of the following system of polynomial equations: 
 \begin{equation} \label{eq:tensor-pr}
  \vx = \alpha \cmP \vx^{(m-1)} + (1-\alpha) \vv 
\quad \text{ or equivalently } \quad
   \vx = \alpha \mR (\fullkron{\vx}{m-1}) + (1-\alpha) \vv
 \end{equation}
 where $\mR := \cmP_{(1)} = \flat_1(\cmP)$ is an $n$-by-$n^{(m-1)}$ stochastic matrix of the flattened tensor along the first index.
\end{definition}

We chose the name \emph{multilinear PageRank} instead of the alternative \emph{tensor PageRank} to emphasize the multilinear structure in the system of polynomial equations rather than the tensor structure of $\cmP$. Also, because the tensor structure of $\cmP$ is shared with the higher-order PageRank vector, which could have then also been called a \emph{tensor PageRank}. 

A multilinear PageRank vector $\vx$ always exists because it is a special case of the stationary distribution vector considered by Li and Ng. In order to apply their theory, we consider the equivalent problem: 
\begin{equation} \label{eq:tpr-ng}
 \vx = (\alpha \mR + (1-\alpha) \vv \ve^\trans) (\vx \kron \cdots \kron \vx) = \cmPbar \vx^{m-1}, 
\end{equation}
where $\cmPbar$ is the \emph{stochastic} transition tensor whose flattening along the first index is the matrix $\alpha \mR + (1-\alpha) \vv \ve^\trans$. The existence of a stochastic solution vector $\vx$ is guaranteed by Brouwer's fixed point theorem, and more immediately, by the stationary distributions considered by Li and Ng. The existing bulk of Perron-Frobenius theory for nonnegative tensors~\cite{Lim-2005-eigenvalues,Chang-2008-perron,Friedland-2013-perron}, unfortunately, is not helpful with existence of uniqueness issues as it applies to problems where $\normof[2]{\vx} = 1$ instead of the 1-norm.

Although the multilinear PageRank vector always exists, it may not be unique as shown in the following example:
\begin{example}
Let $\alpha = 0.99$, $\vv = [0, 1, 0]^\trans$, and 
\[ \mR = \bmat{
                0 &        0     &    0     &    0      &    0     &   0    & 1/3   & 1  &  0 \\
                0 &        0     &    0     &    0      &    1     &    0   & 1/3   & 0  &  1 \\
                1 &         1    &     1    &      1     &     0   &      1 &   1/3 &   0 &   0 }. \]
Then both $\vx = [0, 1, 0]^\trans$ and $\vx = [0.1890, 0.3663, 0.4447]^\trans$ solve the multilinear PageRank problem.
\end{example}

\subsection{A stochastic process: the spacey random surfer} \label{sec:process}
The PageRank vector is equivalently the stationary distribution of the random surfer stochastic process. The multilinear PageRank equation is the stationary distribution of a stochastic process with a history dependent behavior that we call the \emph{spacey random surfer}. For simplicity, we describe this for the case of a second-order problem. Let $S_t$ represent the stochastic process for the spacey random surfer. The process depends on the probability table for a second-order Markov chain $\cmP$. Our motivation is that the spacey surfer would like to transition as the higher-order PageRank Markov chain, $\Pr (S_{t+1} = i \mid S_t = j, S_{t-1} = k ) = \alpha \cP_{ijk} + (1-\alpha) v_i$,  however, on arriving at $S_t = j$, the surfer \emph{spaces out} and \emph{forgets} that $S_{t-1} = k$. Instead of using the true history state, the spacey random surfer decides to \emph{guess} that they came from a state they've visited frequently. Let $Y_t$ be a random state that the surfer visited in the past, chosen according to the frequency of visits to that state. (Hence, $Y_t = k$ is more likely if the surfer visited state $k$ frequently in the past.) The spacey random surfer then transitions as: 
\[ \Pr(S_{t+1} = i \mid S_t = j, Y_{t} = k ) = \alpha \cP_{ijk} + (1-\alpha) v_i. \]

Let us now state the resulting process slightly more formally. Let $\mathcal{F}_t$ be the natural filtration on the history of the process $S_1, \dots, S_t$. Then
\[ \Pr (Y_t = k \mid \mathcal{F}_t ) = \frac{1}{t+n} \left(1 + \sum_{r=1}^{t} \Indof{ S_r = k } \right), \]
where $\Indof{\cdot}$ is the indicator event. In this definition, note that we assume that there is an initial probability of $1/n$ of $Y_t$ taking any state. For instance, if $n=10$ and $S_1 = 5, S_2 = 6, S_3 = 5$ and $t = 3$, then $Y_t$ is a random variable that takes value $6$ with probability $2/13$ and value $5$ probability $3/13$. The stochastic process progresses as: 
\[ \Pr (S_{t+1} = i \mid \mathcal{F}_t ) = \alpha \sum_{k=1}^n \cP(i,S_{t},k) \frac{1 + \sum_{r=1}^{t} \Indof{ S_r = k }}{t+n} + (1-\alpha) v_i. \]
This stochastic process is a new type of vertex reinforced random walk~\cite{Pemantle-1992-vertex-reinforced}. 

We present the following heuristic justification for the equivalence of this process with the multilinear PageRank vector. In our subsequent manuscript~\citet{Gleich-preprint-srs}, we use results from \citet{Benaim-1997-vrrw} to make this equivalence precise and also, to study the process in more depth. Suppose the process has run for a long time $t \gg 1$. Let $\vy$ be the probability distribution of selecting any state as $Y_t$. The vector $\vy$ changes slowly when $t$ is large. For some time in the future, we can approximate the transitions as a first-order Markov chain: 
\[ \Pr(S_{t+1} = i \mid S_t = j) \approx \alpha \cP_{i,j,k} y_k + (1-\alpha) v_i. \] 
Let $\mR_k = \cP(:,:,k)$ be a slice of the probability table, then the Markov transition matrix is: 
\[ \alpha \sum_{k=1}^n \mR_k y_k + (1-\alpha) \vv = \alpha \mR (\vy \kron \mI) + (1-\alpha) \vv \ve^\trans. \]
The resulting stationary distribution is a vector $\vx$ where: 
\[ \vx = \alpha \mR( \vy \kron \vx) + (1-\alpha) \vv . \]
If $\vy = \vx$, then the distribution of $\vy$ will not change in the future, whereas if $\vy \not= \vx$, then the distribution of $\vy$ \emph{must change} in the future. Hence, we must have $\vx = \vy$ at stationarity and any stationary distribution of the spacey random surfer must be a solution of the multilinear PageRank problem.

\subsection{Sufficient conditions for uniqueness}
In this section, we provide a sufficient condition for the multilinear PageRank vector to be unique. Our original conjecture was that this vector would be unique when $\alpha < 1$, which mirrors the case of the standard and higher-order PageRank vectors; however, we have already seen an example where this was false. Throughout this section, we shall derive and prove the following result:
\begin{theorem} \label{thm:tpr-unique}
 Let $\cmP$ be an order-$m$ stochastic tensor, $\vv$ be a nonnegative vector. Then the multilinear PageRank equation 
 \[ \vx = \alpha \cmP \vx^{(m-1)} + (1-\alpha) \vv \]
 has a unique solution when $\alpha < \frac{1}{m-1}$.
\end{theorem}

To prove this statement, we first prove a useful lemma about the norm of the difference of the Kronecker products between two stochastic vectors with respect to the difference of each part. We suspect this result is known, but were unable to find an existing reference.
\begin{lemma}
\label{kronecker_product_norm}
\label{lem:stochastic-diff-2}
Let $\va, \vb, \vc,$ and $\vd$ be stochastic vectors where $\va$ and $\vc$ have the same size. The 1-norm of the difference of their Kronecker products satisfies the following inequality,
\[ \normof[1]{ \va \kron \vb - \vc \kron \vd } \le \normof[1]{\va - \vc} + \normof[1]{\vb - \vd}. \]
\end{lemma}
\begin{proof}
This proof is purely algebraic and begins by observing:
\[ \va \kron \vb - \vc \kron \vd 
	= \frac{1}{2} (\va - \vc) \kron (\vb + \vd) + \frac{1}{2} (\va + \vc) \kron (\vb - \vd). \]
If we separate the bound into pieces we must bound terms such as $\normof[1]{ (\va - \vc) \kron (\vb + \vd) }$. But by using the stochastic property of the vectors, this term equals $\sum_{ij} (b_j + d_j) |a_i - c_i| = 2\normof[1]{\va - \vc}.$
\end{proof}

This result is essentially tight. Let us consider two stochastic vectors of 2 dimensions, $\vx=[x_1, 1-x_1]^\trans$ and $\vy=[y_1, 1-y_1]^\trans$, where $x_1 \neq y_1$.
Then, 
\[
\frac{\normof[1]{\vx \kron \vx - \vy \otimes \vy}}{\normof[1]{\vx - \vy_1}} = \frac{1}{2}|x_1 + y_1| + |1-(x_1+y_1)| + \frac{1}{2}|2 - (x_1+y_1)|.
\]
The ratio of $\normof[1]{\vx \kron \vx - \vy \kron \vy} / \normof[1]{\vx - \vy}$ approaches to 2 as $x_1+y_1 \rightarrow 0$ or $x_1 + y_1 \rightarrow 2$. However, this bound cannot be achieved.

The conclusion of Lemma \ref{kronecker_product_norm} can be easily extended to the case where there are 
multiple Kronecker products between vectors.
\begin{lemma}
\label{lem:stochastic-diff}
For stochastic vectors $\vx_1, \dots, \vx_m$ and $\vy_1, \dots, \vy_m$ where the size of $\vx_i$ is the same as the size of $\vy_i$, then  
\[ \normof[1]{  \vx_1 \kron \cdots \kron \vx_m - \vy_1 \kron \cdots \kron \vy_m } \le \sum_{i} \normof[1]{\vx_i - \vy_i}. \]
\end{lemma}
\begin{proof}
Let us consider the case of $m=3$. Let $\va = \vx_1 \kron \vx_2$, $\vc = \vy_1 \kron \vy_2$, $\vb = \vx_3$ and $\vd = \vy_3$. Then 
\[ \normof[1]{ \vx_1 \kron \vx_2 \kron \vx_3 - \vy_1 \kron \vy_2 \kron \vy_3} = \normof[1]{ \va \kron \vb - \vc \kron \vd} \le \normof[1]{\va - \vc} + \normof[1]{\vx_3 - \vy_3} \]
by using Lemma~\ref{lem:stochastic-diff-2}. But by recurring on $\va - \vc$, we complete the proof for $m=3$. It is straightforward to apply this argument inductively for $m > 3$.
\end{proof} 

This result makes it easy to show uniqueness of the multilinear PageRank vectors:

\begin{lemma}
The multilinear PageRank equation has the unique solution when $\alpha < 1/2$ for third order tensors.
\end{lemma}
\begin{proof}
Assume there are two distinct solutions to the multilinear PageRank equation,
\[
\begin{aligned}
\vx &= \alpha \mR(\vx \otimes \vx) + (1-\alpha)\vv\\
\vy &= \alpha \mR(\vy \otimes \vy) + (1-\alpha)\vv \\
\vx - \vy &= \alpha \mR (\vx \otimes \vx - \vy \otimes \vy).
\end{aligned}
\]
We simply apply Lemma~\ref{lem:stochastic-diff-2}:
\[
\normof[1]{\vx - \vy} = \normof[1]{\alpha \mR (\vx \otimes \vx - \vy \otimes \vy)} \le 2\alpha\normof[1]{\mR} \normof[1]{\vx-\vy} < \normof[1]{\vx-\vy},
\]
which is a contradiction (recall that $\mR$ is stochastic). Thus, the multilinear PageRank equation has the unique solution when $\alpha < 1/2$.
\end{proof} \bigskip

The proof of the general result in Theorem~\ref{thm:tpr-unique} is identical, except that it uses the general bound  Lemma~\ref{lem:stochastic-diff} on the order-$m$ problem.

\subsection{Uniqueness via Li and Ng's results}

Li and Ng's recent paper~\cite{Li-2013-tensor-markov-chain} tackled the same uniqueness question for the general problem: \[ \cmP \vx^{m-1} = \vx. \]  We can also write our problem in this form as in equation~\ref{eq:tpr-ng} and apply their theory. In the case of a third-order problem, or $m=3$, they define a quantity to determine uniqueness:
\[ \bigbeta. \] 
For any tensor where $\beta > 1$, the vector $\vx$ that solves \[ \cmP \vx^2 = \vx \] is unique. In Appendix~\ref{sec:pagerank-beta} to this paper, we show that $\beta > 1$ is a stronger condition that $\alpha < 1/2$. We defer this derivation to the appendix as it is slightly tedious and does not result in any new insight into the problem.

\subsection{PageRank and higher-order PageRank}

We conclude this section by establishing some relationships between multilinear PageRank, higher-order PageRank, and PageRank for a special tensor. In the case when there is no higher-order structure present, then the multilinear PageRank, higher-order PageRank, and PageRank are all equivalent. The precise condition is where $\mR = \ve^\trans \kron \mQ$ for a stochastic matrix $\mQ$, which models a higher-order random surfer with behavior that is \emph{independent} of the last state. Thus, we'd expect that none of our higher-order modifications would change the properties of the stationary distribution. 

\begin{proposition}
Consider a second-order multilinear PageRank problem with a third-order stochastic tensor where the flattened matrix $\mR = \ve^\trans \kron \mQ$ has dimension $n \times n^2$ and where $\mQ$ is an $n \times n$ column stochastic matrix. Then for all $0 < \alpha < 1$ and stochastic vectors $\vv$, the multilinear PageRank vector is the same as the PageRank vector of $(\alpha, \mQ, \vv)$. Also, the marginal distribution of the higher-order PageRank solution matrix, $\mX \ve$, is the same as well.
\end{proposition}
\begin{proof}
If $\mR = \ve^\trans \otimes \mQ$, then any solution of equation~\eqref{eq:tensor-pr} is also the unique solution of the standard PageRank equation:
\[
\vx = \alpha (\ve^\trans \kron \mQ)(\vx \kron \vx) + (1-\alpha)\vv = \alpha \mQ \vx + (1-\alpha)\vv.
\]
Thus, the two solutions must be the same and the multilinear PageRank problem has a unique solution as well. 
Now consider the solution of the second-order PageRank problem from equation~\eqref{eq:1-step chain}:
\[ \tvec(\mX) = [\alpha \mP + (1-\alpha) \mV ]\tvec(\mX). \]
Note that $\mR = (\ve^\trans \kron \mI) \mP$. Consider the marginal distribution: $\vy = \mX \ve = (\ve^\trans \kron \mI) \tvec(\mX)$. The vector $\vy$ must satisfy: 
\[ \vy = (\ve^\trans \kron \mI) \tvec(\mX) = (\ve^\trans \kron \mI)  [\alpha \mP + (1-\alpha) \mV ]\tvec(\mX) = \alpha \mR \tvec(\mX) + (1-\alpha) \vv.\]
But $\mR \tvec(\mX) = (\ve^\trans \kron \mQ) \tvec(\mX) = \mQ \vy$.
\end{proof}
 
\subsection{Fast operators from sparse data}
\label{sec:fast-operators}
The last detail we wish to mention is how to build a fast operator $\cmP \vx^{m-1}$ when the input tensor is highly sparse. Let $\cmQ$ be the tensor that models the original \emph{sparse} data, where $\cmQ$ has far fewer than $n^{m-1}$ non-zeros and cannot be stochastic. Nevertheless, suppose that $\cmQ$ has the following property: 
\[ \cQ(i, j, \dots, k) \ge 0 \quad \text{ and } \quad \sum_i \cQ(i, j, \dots, k) \le 1 \text{ for all } j, \dots, k. \]
This could easily be imposed on a set of nonnegative data in time and memory proportional to the non-zeros of $\cmQ$ if that were not originally true. To create a fast operator for a fully stochastic problem, we generalize the idea behind the \emph{dangling indicator} correction of PageRank. (The following derivation is entirely self contained, but the genesis of the idea is identical to the dangling correction in PageRank problems~\cite{boldi2007-traps}.) Let $\mS$ be the flattening of $\cmQ$ along the first index. Let $\vd^\trans = \ve^\trans - \ve^\trans \mS \ge 0$, and let $\vu$ be a stochastic vector that determines what the model should do on a \emph{dangling case}. Then: 
\[ \mR = \mS + \vu \vd^\trans \]
is a column stochastic matrix, which we interpret as the flattening of $\cmP$ along the first index. If $\vx$ is a stochastic vector, then we can evaluate: 
\[ \mR \vx = \underbrace{\mS \vx}_{\vz} + \vu (\ve^\trans \vx - \ve^\trans \mS \vx) = \vz + (1 - \ve^\trans \vz) \vu, \]
which only involves work proportional to the non-zeros of $\mS$ or the non-zeros of $\mQ$.
Thus, given any sparse tensor data, we can create a fully stochastic model.

\section{Algorithms for Multilinear PageRank} \label{sec:methods}

At this point, we begin our discussion of algorithms to compute the multilinear PageRank vector. In the following
section, we investigate five different methods to compute it. The methods are all inspired by the fixed-point
nature of the multilinear PageRank solution. They are: 
\begin{compactenum}
\item a fixed-point iteration, as in the power method and Richardson method;
\item a shifted fixed-point iteration, as in SS-HOPM~\cite{Kolda-2011-sshopm};
\item a non-linear inner-outer iteration, akin to \citet{gleich2010-inner-outer};
\item an inverse iteration, as in the inverse power method; and
\item a Newton iteration.
\end{compactenum}
We will show that the first four of them converge in the case that $\alpha < 1/(m-1)$ for an order-$m$  tensor. For Newton, we show it converges quadratically fast for a third-order tensor when $\alpha < 1/2$. We also illustrate a few empirical advantages of each method.

\paragraph{The test problems}
Throughout the following section, the following two problems help illustrate the methods:
\[ \begin{aligned}
\mR_1 & = \left[ \begin{array}{ccc|ccc|ccc}
1/3 & 1/3 & 1/3 & 1/3 & 0 & 0 & 0 & 0 & 0 \\
1/3 & 1/3 & 1/3 & 1/3 & 0 & 1/2 & 1 & 0 & 1 \\
1/3 & 1/3 & 1/3 & 1/3 & 1 & 1/2 & 0 & 1 & 0 \\
\end{array} \right] \\
\mR_2 & = \left[ \begin{array}{cccc|cccc|cccc|cccc}
0 & 0 & 0 & 0 & 0 & 0 & 0 & 0 & 0 & 0   & 0 & 0 & 0 & 0   & 0 & 1/2 \\
0 & 0 & 0 & 0 & 0 & 1 & 0 & 1 & 0 & 1/2 & 0 & 0 & 0 & 1/2 & 0 & 0 \\
0 & 0 & 0 & 0 & 0 & 0 & 1 & 0 & 0 & 1/2 & 1 & 1 & 0 & 0   & 0 & 0 \\
1 & 1 & 1 & 1 & 1 & 0 & 0 & 0 & 1 & 0   & 0 & 0 & 1 & 1/2 & 1 & 1/2 \\
\end{array} \right]
\end{aligned} \]
with $\vv = \ve/n$. The parameter $\alpha$ will vary through our experiments, but we are most interested in the regime where $\alpha > 1/2$ to understand how the algorithms behave outside of the region where we can prove they converge.
We derived these problems by using exhaustive and randomized searches over the
space of $2 \times 2 \times 2$, $3 \times 3 \times 3$, and $4 \times 4 \times 4$
binary-valued tensors, which we then normalized to be stochastic. Problems $\mR_1$ and $\mR_2$ were made from the database of problems we consider from the next section (Section~\ref{sec:experiments}).

The residual of a problem and a potential solution $\vx$ is the $1$-norm:
\begin{equation} \label{eq:residual}
 \text{residual} = \normof[1]{ \alpha \cmP \vx^{m-1} + (1-\alpha) \vv - \vx}. 
\end{equation}
We seek methods that cause the residual to drop below $10^{-8}$. For all the examples in this section, we ran the method out to $10,000$ iterations to ensure there was no delayed convergence, although, we only show $1000$ iterations.

\subsection{The fixed-point iteration}

The multilinear PageRank problem seeks a fixed-point of the following non-linear map: 
\[ f(\vx) = \alpha \cmP \vx^{m-1} + (1-\alpha) \vv. \]
We first show convergence of the iteration implied by this map in the case
that $\alpha < 1/(m-1)$.
\begin{theorem}
Let $\cmP$ be an order-$m$ stochastic tensor, let $\vv$ and $\vx\itn{0}$ be stochastic vectors, and let $\alpha < 1/(m-1)$. The fixed-point iteration 
\[ \vx\itn{k+1}^{} = \alpha \cmP \vx\itn{k}^{m-1} + (1-\alpha) \vv \]
will converge to the unique solution $\vx$ of the multilinear PageRank problem~\eqref{eq:tensor-pr} and also 
\[ \normof[1]{\vx\itn{k} - \vx} \le [\alpha(m-1)]^k \normof[1]{\vx\itn{0} - \vx} \le 2[\alpha(m-1)]^k. \]
\end{theorem}
\begin{proof}
Note first that this problem has a unique solution $\vx$ by Theorem~\ref{thm:tpr-unique}, and also that $\vx\itn{k}$ remains stochastic for all iterations. 
This result is then, essentially, an implication of Lemma~\ref{lem:stochastic-diff}.
Let $\mR$ be the flattening of $\cmP$ along the first index. 
Then using that Lemma, 
\[ \normof[1]{\vy - \vx\itn{k+1}} = \normof[1]{\alpha \mR (\vy \kron \cdots \kron \vy - \vx\itn{k} \kron \cdots \kron \vx\itn{k})} \le \alpha (m-1) \normof[1]{\vy - \vx\itn{k}}. \]
Thus, we have established a contraction.
\end{proof}

Li and Ng treat the same iteration in their paper and they show a more general convergence result that implies our theorem, thus providing a more refined understanding of the convergence of this iteration. However, their result needs a difficult-to-check criteria. In earlier work by~\cite{Rabinovich-1992-quadratic}, they show that the fixed-point iteration will always converge when a certain symmetry property holds, however, they do not have a rate of convergence. Nevertheless, it is still easy to find PageRank problems that will not converge with this iteration. Figure~\ref{fig:fp-diverge} shows the result of using this method on $\mR_1$ with $\alpha = 0.95$ and $\alpha = 0.96$. The former converges nicely and the later does not.

\begin{figure}
\includegraphics[width=0.5\linewidth]{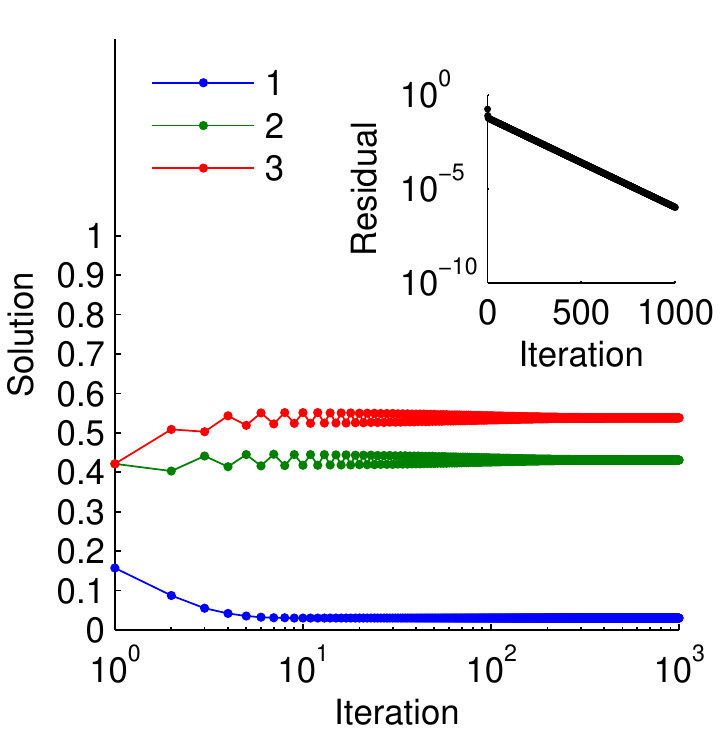}%
\includegraphics[width=0.5\linewidth]{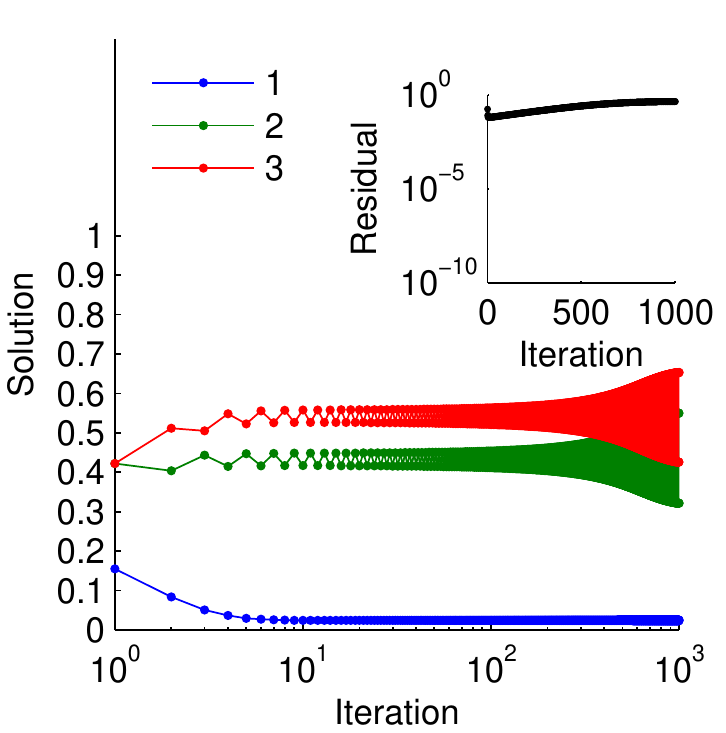}%

\caption{At left, the components of the iterates of the fixed point iteration for $\mR_1$ with $\alpha = 0.95$ show that it converges to a solution. (See the inset residual in the upper right.) At right, the result of the fixed point iteration with $\alpha=0.96$ illustrates a case that does not converge. }
\label{fig:fp-diverge}
\end{figure}

\subsection{The shifted fixed-point iteration}
\citet{Kolda-2011-sshopm} noticed a similar phenomenon for the convergence of the symmetric higher-order power method and proposed the \emph{shifted} symmetric higher-order power method (SS-HOPM) to address these types of oscillations. They were able to show that their iteration always converges monotonically for an appropriate shift value. For the multilinear PageRank problem, we study the iteration given by the equivalent fixed-point: 
\[ (1+\gamma) \vx = \left[ \alpha \cmP \vx^{m-1} + (1-\alpha) \vv \right] + \gamma \vx. \]
The resulting iteration is what we term the \emph{shifted fixed-point iteration}
\[
 \vx\itn{k+1}^{} = \frac{\alpha}{1+\gamma} \cmP \vx\itn{k}^{m-1} + \frac{1-\alpha}{1+\gamma} \vv + \frac{\gamma}{1+\gamma} \vx\itn{k}. 
 \]
It shares the property that an initial stochastic approximation $\vx\itn{0}$ will remain stochastic throughout. 

\begin{theorem}
Let $\cmP$ be an order-$m$ stochastic tensor, let $\vv$ and $\vx\itn{0}$ be stochastic vectors, and let $\alpha < 1/(m-1)$. The shifted fixed-point iteration 
\begin{equation} \label{eq:shifted}
 \vx\itn{k+1}^{} = \frac{\alpha}{1+\gamma} \cmP \vx\itn{k}^{m-1} + \frac{1-\alpha}{1+\gamma} \vv + \frac{\gamma}{1+\gamma} \vx\itn{k} 
\end{equation}
will converge to the unique solution $\vx$ of the multilinear PageRank problem~\eqref{eq:tensor-pr} and also 
\[ \normof[1]{\vx\itn{k} - \vx} \le \left( \frac{\alpha(m-1) + \gamma}{1+\gamma} \right)^k \normof[1]{\vx\itn{0} - \vx} \le 2\left( \frac{\alpha(m-1) + \gamma}{1+\gamma} \right)^k. \]
\end{theorem}
The proof of this convergence is, in essence, identical to the previous case and we omit it for brevity. 

This result also suggests that choosing $\gamma = 0$ is optimal and we should not shift the iteration at all. That is, we should run the fixed-point iteration. This analysis, however, is misleading as illustrated in Figure~\ref{fig:shifted}. There, we show the iterates from solving $\mR_1$ with $\alpha = 0.96$, which did not converge with the fixed-point iteration, but converges nicely with $\gamma = 1/2$. However, $\gamma < (m-2)/2$ will not guarantee convergence and the same figure shows that $\mR_2$ with $\alpha = 0.97$ will not converge. We now present a necessary analysis that shows this method may not converge if $\gamma < (m-2)/2$ when $\alpha > 1/(m-1)$.

\begin{figure}
\includegraphics[width=0.5\linewidth]{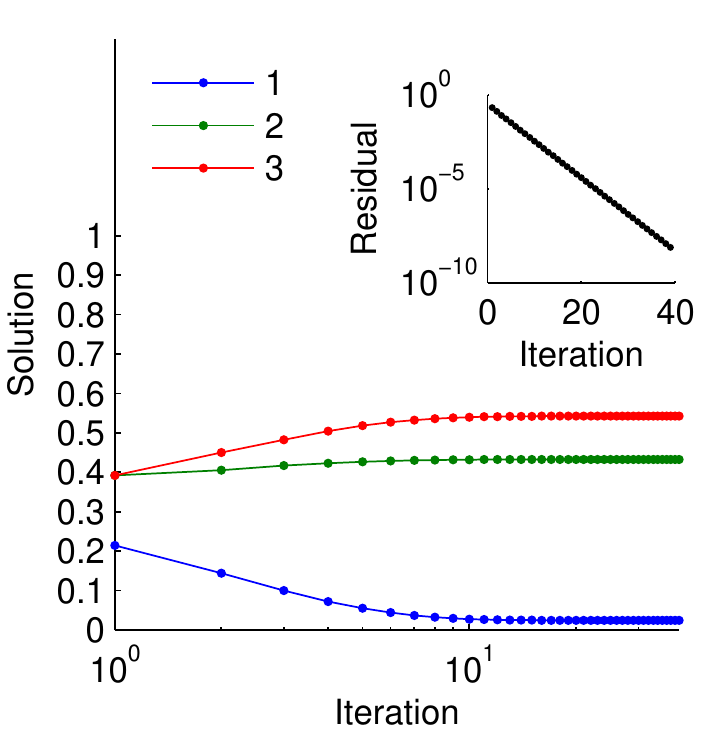}%
\includegraphics[width=0.5\linewidth]{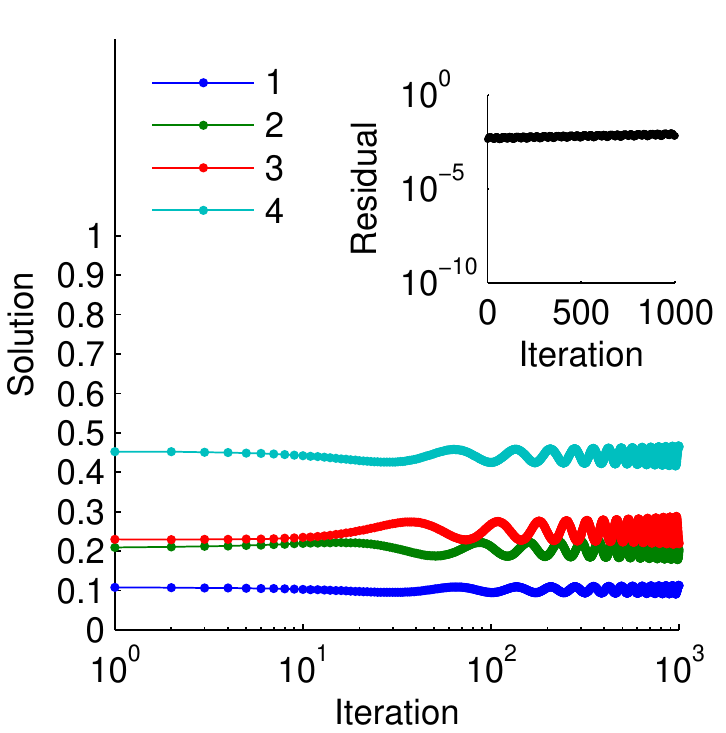}%

\caption{When we use a shift $\gamma = 1/2$, then at left, the iterates of the shifted iteration for $\mR_1$ with $\alpha = 0.96$ shows that it quickly converges to a solution, whereas this same problem did not converge with the fixed-point method.  At right, the result of the shifted iteration on $\mR_2$ with $\alpha=0.97$ again shows a case that does not converge. }
\label{fig:shifted}
\end{figure}

\paragraph{On the necessity of shifting} To derive this result, we shall restate the multilinear PageRank problem as the limit point of an ordinary differential equation. There are other ways to derive this result as well, but this one is familiar and relatively straightforward. Consider the ordinary differential equation: 
\begin{equation} \label{eq:ode}
 \frac{d\vx}{dt} = \alpha \cmP \vx^{m-1} + (1-\alpha) \vv - \vx. 
\end{equation}
A forward Euler discretization yields the iteration: 
\[ \vx\itn{k+1}^{} = \alpha h \cmP \vx\itn{k}^{m-1} + (1-\alpha) h \vv + (1-h) \vx\itn{k}, \]
which is identical to the shifted iteration~\eqref{eq:shifted} with $h = \frac{1}{1+\gamma}$. To determine if forward Euler converges, we need to study the Jacobian of the ordinary differential equation. Let $\mR$ be the flattening of $\cmP$ along the first index, then the Jacobian of the ODE \eqref{eq:ode} is: 
\[ \mJ(\vx) = \alpha \mR ( \mI \kron \vx \kron \cdots \kron \vx + \vx \kron \mI \kron \vx \kron \cdots \kron \vx + \vx \kron \cdots \kron \vx \kron \mI) - \mI. \] 
A necessary condition for the forward Euler method to converge is that it is absolutely stable. In this case, we need  $|1-h \rho(\mJ)| \le 1$, where $\rho$ is the spectral radius of the Jacobian. For all stochastic vectors $\vx$ generated by iterations of the algorithm, $\rho(\mJ(\vx)) \le (m-1) \alpha + 1 \le m$. Thus, $h \le 2/m$ is necessary for a general convergence result when $\alpha > 1/(m-1)$ . This, in turn, implies that $\gamma \ge (m-2)/2$. In the case that $\alpha < 1/(m-1)$, then the Jacobian already has eigenvalues within the required bounds and no shift is necessary.

\begin{remark}
Based on this analysis, we always recommend the shifted iteration with $\gamma \ge (m-2)/2$ for any problem with $\alpha > 1/(m-1)$.
\end{remark}


\subsection{An inner-outer iteration}

We now develop a non-linear iteration scheme using that uses multilinear PageRank, in the convergent regime, as a subroutine. To derive this method, we use the relationship between multilinear PageRank and the multilinear Markov chain formulation discussed in Section~\ref{sec:li}. Let $\mRbar  = \alpha \mR + (1-\alpha) \vv \ve^\trans$ then note that this the Markov chain form of the problem is: 
\[ \mRbar (\fullkron{\vx}{m-1}) = \vx \quad \Leftrightarrow \quad [ \alpha \mR + (1-\alpha) \vv \ve^\trans ] (\fullkron{\vx}{m-1}) = \vx. \] 
Equivalently, we have: 
\[ \frac{\alpha}{m-1} \mRbar (\fullkron{\vx}{m-1}) + \left(1-\frac{\alpha}{m-1}\right) \vx = \vx. \]
From here, the nonlinear iteration emerges: 
\begin{equation} \vx\itn{k+1} = \frac{\alpha}{m-1} \mRbar (\fullkron{\vx\itn{k+1}}{m-1}) + \left(1-\frac{\alpha}{m-1}\right) \vx\itn{k}. \end{equation}
Each iteration involves solving a multilinear PageRank problem with $\mRbar, \alpha/(m-1),$ and $\vx\itn{k}$. Because $\alpha < 1$, then $\alpha/(m-1) < 1/(m-1)$ and the solution of these subproblems is unique, and thus, the method is well-defined. Not surprisingly, this method also converges when $\alpha < 1/(m-1)$.

\begin{theorem}
Let $\cmP$ be an order-$m$ stochastic tensor, let $\vv$ and $\vx\itn{0}$ be stochastic vectors, and let $\alpha < 1/(m-1)$. Let $\mR$ be the flattening of $\cmP$ along the first index and let $\mRbar = \alpha \mR + (1-\alpha) \vv \ve^\trans$.
 The inner-outer multilinear PageRank iteration
\[ \vx\itn{k+1} = \frac{\alpha}{m-1} \mRbar (\fullkron{\vx\itn{k+1}}{m-1}) + \left(1-\frac{\alpha}{m-1}\right) \vx\itn{k} \]
 converges to the unique solution $\vx$ of the multilinear PageRank problem and also 
 \[ \normof[1]{\vx\itn{k} - \vx} \le \left( \frac{1-\alpha/(m-1)}{1-\alpha^2} \right)^{k} \normof[1]{\vx\itn{0} - \vx} \le 2 \left( \frac{1-\alpha/(m-1)}{1-\alpha^2} \right)^k. \]
\end{theorem}
\begin{proof}
Recall that this is the regime of $\alpha$ when the solution is unique. Note that 
\[ \begin{aligned}
\vx\itn{k+1} - \vx 
& = \frac{\alpha}{m-1} \mRbar ( \fullkron{\vx\itn{k+1}}{m-1} - \fullkron{\vx}{m-1} ) + \left(1 - \frac{\alpha}{m-1}\right) (\vx\itn{k} - \vx) \\
& = \frac{\alpha^2}{m-1} \mR ( \fullkron{\vx\itn{k+1}}{m-1} - \fullkron{\vx}{m-1} ) + \left(1 - \frac{\alpha}{m-1}\right) (\vx\itn{k} - \vx). 
\end{aligned} \]
By using Lemma~\ref{lem:stochastic-diff}, we can bound the norm of the difference of the $m-1$ term Kronecker products by $(m-1) \normof[1]{\vx\itn{k+1} - \vx}$. Thus, 
\[ \normof[1]{\vx\itn{k+1} - \vx} \le \alpha^2 \normof[1]{\vx\itn{k+1} - \vx} + \left(1-\frac{\alpha}{m-1}\right) \normof[1]{\vx\itn{k} - \vx}, \]
and the scheme converges linearly with rate $ \frac{1-\alpha/(m-1)}{1-\alpha^2} < 1$ when $\alpha < 1/(m-1)$.
\end{proof}

In comparison with the shifted method, each iteration of the inner-outer method is far more expensive and involves solving a multilinear PageRank method. However, if $\cmP$ is only available through a fast operator, this may be the only method possible. In Figure~\ref{fig:innout}, we show that the inner-outer method converges in the case that the shifted method failed to converge. Increasing $\alpha$ to $0.99$, however, now generates a problem where the inner-outer method will not converge.

\begin{figure}
\includegraphics[width=0.5\linewidth]{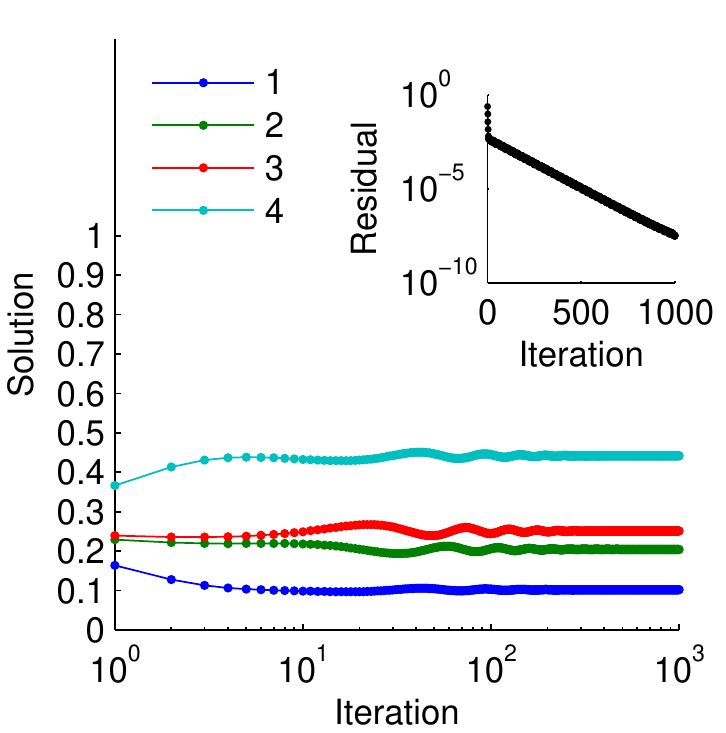}%
\includegraphics[width=0.5\linewidth]{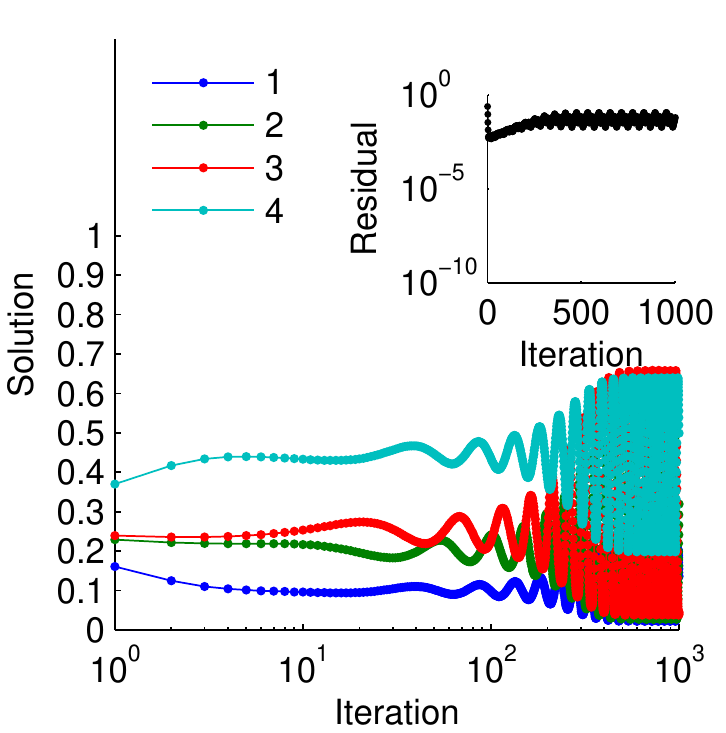}%

\caption{At left, the iterates of the inner-outer iteration for $\mR_2$ with $\alpha = 0.97$ shows that it converges to a solution, whereas this same problem did not converge with the shifted method. At right, the result of the shifted iteration on $\mR_2$ with $\alpha=0.90$ again shows an example that doesn't converge. }
\label{fig:innout}
\end{figure}

\subsection{An inverse iteration}

Another algorithm we consider is given by our interpretation of the multilinear PageRank solution as a stochastic process.  Observe, for the second-order case, 
\[ \alpha \mR (\vx \kron \vx) = \tfrac{\alpha}{2}\mR (\vx \kron \mI + \mI \kron \vx) = \alpha \, [ \tfrac{1}{2} \mR (\vx \kron \mI) + \tfrac{1}{2} \mR ( \mI \kron \vx) ]. \]
Both matrices 
\[ \mR (\vx \kron \mI) \quad \text{ and } \quad \mR (\mI \kron \vx) \]
are stochastic. Let $\mS(\vx) = \tfrac{1}{2} \mR (\vx \kron \mI) + \tfrac{1}{2} \mR ( \mI \kron \vx)$ be stochastic sum of these two matrices. Then the multilinear PageRank vector satisfies: 
\[ \vx = \alpha \mS(\vx) \vx + (1-\alpha) \vv. \]
This equation has a subtle interpretation. The multilinear PageRank vector \emph{is} the PageRank vector of a solution dependent Markov process. The stochastic process presented in Section~\ref{sec:process} shows this in a slightly different manner. The iteration that arises is a simple fixed-point idea using this interpretation: 
\[ \vx\itn{k+1} = \alpha \mS(\vx\itn{k}) \vx\itn{k+1} + (1-\alpha) \vv. \]
Thus, at each step, we solve a PageRank problem given the current iterate to produce the subsequent vector. For this iteration, we could then leverage a fast PageRank solver if there is a way of computing $\mS(\vx\itn{k})$ effectively or $\mS(\vx\itn{k}) \vx$ effectively. The method for a general problem is the same, except for the definition of $\mS$. In general, let
\begin{equation} \label{eq:mlpr-markov}
 \mS(\vx) = \tfrac{1}{m-1} \mR (\mI \kron \fullkron{\vx\itn{k}}{m-2} + \vx\itn{k} \kron \mI \kron \fullkron{\vx\itn{k}}{m-3} + \dots + \fullkron{\vx\itn{k}}{m-2} \kron \mI). 
 \end{equation}
This iteration is guaranteed to converge in the unique solution regime.

\begin{theorem}
Let $\cmP$ be an order-$m$ stochastic tensor, let $\vv$ and $\vx\itn{0}$ be stochastic vectors, and let $\alpha < 1/(m-1)$. Let $\mS(\vx\itn{k})$ be an $n \times n$ stochastic matrix defined via~\eqref{eq:mlpr-markov}. The inverse multilinear PageRank iteration 
\[ \vx\itn{k+1} = \alpha \mS(\vx\itn{k}) \vx\itn{k+1} + (1-\alpha) \vv \]
converges to the unique solution $\vx$ of the multilinear PageRank problem and also 
\[ \normof[1]{\vx\itn{k} - \vx} \le  \left( \frac{(m-2) \alpha}{1-\alpha} \right)^k \normof[1]{\vx\itn{0} - \vx} \le 2 \left( \frac{(m-2) \alpha}{1-\alpha} \right)^k. \]
\end{theorem}
\begin{proof} We complete the proof using the terms involved in the fourth-order case ($m=4$) because it simplifies the indexing tremendously although the terms in our proof will be entirely general. Consider the error at the $(k+1)$\textsuperscript{th} iteration: 
\[\begin{aligned}
 \vx\itn{k+1} - \vx & = \frac{\alpha}{m-1} \mR [ (\mI \kron \vx\itn{k} \kron \vx\itn{k}  + \vx\itn{k} \kron \mI \kron \vx\itn{k} + \vx\itn{k} \kron \vx\itn{k} \kron \mI ) \vx\itn{k+1} \\
 & \qquad \qquad \qquad - (\mI \kron \vx \kron \vx + \vx \kron \mI \kron \vx + \vx \kron \vx \kron \mI ) \vx ] \\ 
  & = \frac{\alpha}{m-1} \mR [ (\vx\itn{k+1} \kron \vx\itn{k} \kron \vx\itn{k}  + \vx\itn{k} \kron \vx\itn{k+1} \kron \vx\itn{k} + \vx\itn{k} \kron \vx\itn{k} \kron \vx\itn{k+1} ) \\
  & \qquad \qquad \qquad - ( \vx \kron \vx \kron \vx + \vx \kron \vx \kron \vx + \vx \kron \vx \kron \vx). ]
  \end{aligned} \]
At this point, it suffices to prove that terms of the form $\normof[1]{ \vx\itn{k+1} \kron \vx\itn{k} \kron \vx\itn{k} - \vx \kron \vx \kron \vx } $ are bounded by $ ( m-1) \normof[1]{ \vx\itn{k+1} - \vx} + (m-1)(m-2) \normof[1]{\vx\itn{k} - \vx}$. Showing this for one term also suffices because all of these terms are equivalent up to a permutation.

We continue by Lemma~\ref{lem:stochastic-diff}, which yields 
\[ \normof[1]{ \vx\itn{k+1} \kron \vx\itn{k} \kron \vx\itn{k} - \vx \kron \vx \kron \vx } \le \normof[1]{ \vx\itn{k+1} - \vx} + 2 \normof[1]{ \vx\itn{k} - \vx}  \]
in the third-order case, and $\normof[1]{ \vx\itn{k+1} - \vx} + (m-2) \normof[1]{ \vx\itn{k} - \vx}$ in general. Since there are $m-1$ of these terms, we are done.
\end{proof}

In comparison to the inner-outer iteration, this method requires detailed knowledge of the operator $\cmP$ in order to form $\mS(\vx_k)$ or even matrix-vector products $\mS(\vx_k) \vz$. In some applications this may be easy. In Figure~\ref{fig:inverse}, we illustrate the convergence of the inverse iteration on the problems that the inner-outer method's illustration used. The convergence pattern is the same. 

\begin{figure}
\includegraphics[width=0.5\linewidth]{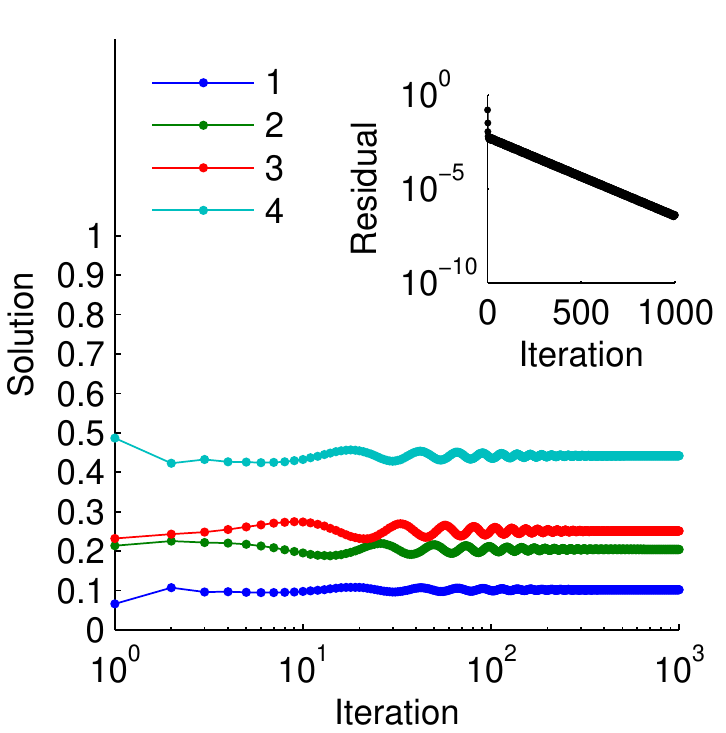}%
\includegraphics[width=0.5\linewidth]{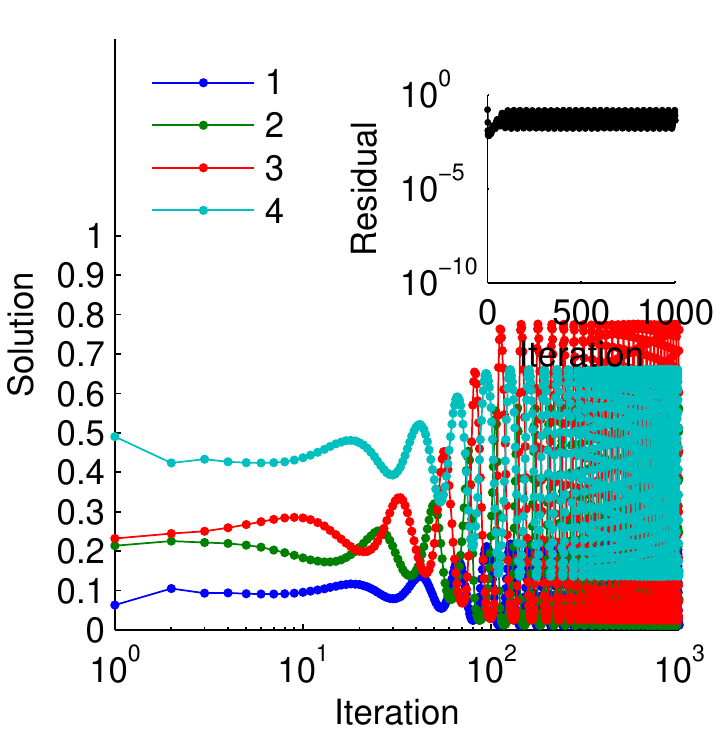}%

\caption{At left, the iterates from the inverse iteration to solve problem $\mR_2$ with $\alpha = 0.97$ and at right, the iterates to solve problem $\mR_2$ with $\alpha = 0.99$. Both show similar convergence behavior as to the inner-outer method.}
\label{fig:inverse}
\end{figure}

\subsection{Newton's method}
Finally, consider Newton's method for solving the nonlinear equation: 
\[ \vf(\vx) = \alpha \mR (\vx \kron \vx) + (1-\alpha) \vv - \vx = 0. \]
The Jacobian of this operator is: 
\[ \mJ(\vx) = \alpha \mR (\vx \kron \mI + \mI \kron \vx) - \mI. \]

We now prove the following theorem about the convergence of Newton's method.

\begin{theorem} \label{thm:newton}
Let $\cmP$ be a third-order stochastic tensor, let $\vv$ be a stochastic vector, and let $\alpha < 1/2$. Let $\mR$ be the flattening of $\cmP$ along the first index. Let 
\[ \vf(\vx) = \alpha \mR (\vx \kron \vx) + (1-\alpha) \vv - \vx = 0. \] 
Newton's method to solve $\vf(\vx) = 0$, and hence compute the unique multilinear PageRank vector, is the iteration:
\begin{equation} \label{eq:newton}
 \left[ \mI - \alpha \mR (\vx\itn{k} \kron \mI + \mI \kron \vx\itn{k}) \right] \vp\itn{k} = \vf(\vx\itn{k}) \qquad \vx\itn{k+1} = \vx\itn{k} + \vp\itn{k} \qquad \vx\itn{0} = 0.
\end{equation}
It produces a unique sequence of iterates where: 
\[ \vf(\vx\itn{k}) \ge 0 \quad \ve^\trans \vf(\vx\itn{k}) = \frac{\alpha (\ve^\trans \vf(\vx\itn{k-1}))^2}{(1-2\alpha)^2 + 4 \alpha \ve^\trans \vf(\vx\itn{k-1})} \le \alpha (1-\alpha)^2 \frac{1}{4^{k-1}} \quad k \ge 1  \]
that also converges quadratically in the $k\to\infty$ limit.
\end{theorem}

This result shows that Newton's method always converges quadratically fast when solving multilinear PageRank vectors inside the unique regime.

\begin{proof}
 We outline the following sequence of facts and lemmas we provide to compute the result. The key idea is to use the result that second-order multilinear PageRank is a 2nd-degree polynomial, and hence, we can use Taylor's theorem to derive an \emph{exact} prediction of the function value at successive iterations. We first prove this \emph{key fact}. Subsequent steps of the proof establish that the sequence of iterates is unique and well-defined (that is, that the Jacobian is always non-singular). This involves showing, additionally, that $\vx\itn{k} \ge 0$, $\ve^\trans \vx\itn{k} \le 1$, and $\vf(\vx\itn{k}) \ge 0$. Let $f_k = \ve^\trans \vf$, since $\vf(\vx\itn{k}) \ge 0$, showing that $f_k \to 0$ suffices to show convergence. Finally, we derive a recurrence: \begin{equation} \label{eq:newton-recur} f_{k+1} = \frac{\alpha f_k^2}{(1-2\alpha)^2 + 4 \alpha f_k}. \end{equation}
 
\noindent \textbf{Key fact.}  Let $\vp\itn{k} = \vx\itn{k+1} - \vx\itn{k}$. If the Jacobian $\mJ(\vx\itn{k})$ is non-singular in the $k$\textsuperscript{th} iteration, 
then $\vf(\vx\itn{k+1}) = \alpha \mR (\vp\itn{k} \kron \vp\itn{k})$. To prove this, we use an exact version of Taylor's theorem around the point $\vx\itn{k}$: 
\[ \vf(\vx\itn{k} + \vp) = \vf(\vx\itn{k}) + \mJ(\vx\itn{k}) \vp + \tfrac{1}{2} \cmT \vp^2, \]
where $\cmT \vp^2 = \alpha \mR (\mI \kron \mI + \mI \kron \mI) (\vp \kron \vp)$ is \emph{independent} of the current point. Note also that Newton's method chooses $\vp$ such that $\vf(\vx\itn{k}) + \mJ(\vx\itn{k}) \vp = 0$.  Then
\[ \vf(\vx\itn{k+1}) =  \vf(\vx\itn{k} + \vp\itn{k}) = \alpha \mR (\vp\itn{k} \kron \vp\itn{k}). \]

\noindent \textbf{Well-defined sequence.} We now show that $\mJ(\vx\itn{k})$ is non-singular for all $\vx\itn{k}$, and hence, the Newton iteration is well-defined. It is easy to do so if we establish that
\begin{equation} \label{eq:newton-props}
 \vf(\vx\itn{k}) \ge 0, \vx\itn{k} \ge 0, \text{ and } z_k = \ve^\trans \vx\itn{k} \le 1 
\end{equation}
also holds at each iteration. Clearly, these properties hold for the initial iteration where $\vf(\vx\itn{0}) = (1-\alpha) \vv$. Thus, we proceed inductively.  Note that if $\vx\itn{k} \ge 0$ and $z_k \le 1$ then the Jacobian is non-singular because $-\mJ(\vx\itn{k}) = [\mI - \alpha \mR ( \vx\itn{k} \kron \mI + \mI \kron \vx\itn{k} ) ]$ is a strictly diagonally dominant matrix, $M$-matrix when $\alpha < 1/2$. (In fact, both $\mR ( \vx\itn{k} \kron \mI)$ and $\mR (\mI \kron \vx\itn{k})$ are nonnegative matrices with column norms equal to $z_k = \ve^\trans\vx\itn{k}$.) Thus, $\vx\itn{k+1}$ is well-defined and it remains to show that $\vf(\vx\itn{k+1}) \ge 0$, $\vx\itn{k+1} \ge 0$, and $z_{k+1} \le 1$. Now, by the definition of Newton's method: 
\[ \vx\itn{k+1} = \vx\itn{k} - \mJ(\vx\itn{k})^{-1} \vf(\vx\itn{k}), \] but $-\mJ$ is an $M$-matrix, and so $\vx\itn{k+1} \ge 0$. This also shows that $\vp\itn{k} = \vx\itn{k+1} - \vx\itn{k} \ge 0$, from which, we can use our key fact to derive that $\vf(\vx\itn{k+1}) \ge 0$. What remains to show is that $z_{k+1} \le 1$. By taking summations on both sides of \eqref{eq:newton}, we have: 
\[ (1-2\alpha z_k) (z_{k+1} - z_k) = \alpha z_k^2 + (1-\alpha) - z_k. \]
A quick, but omitted, calculation confirms that $z_{k+1} > 1$ implies $z_{k} > 1$. Thus, we completed our inductive conditions for \eqref{eq:newton-props}.

\noindent \textbf{Recurrence} We now show that \eqref{eq:newton-recur} holds. First, observe that 
\[ f_k = \alpha (\ve^\trans \vp\itn{k})^2 \qquad \ve^\trans \vp\itn{k} = \frac{f_k}{1-2 \alpha z_k} \qquad \alpha z_k^2 + (1-\alpha) - z_k - f_k = 0. \]
We now solve for $z_k$ in terms of $f_k$. This involves picking a root for the quadratic equation. Since $z_k \le 1$, this makes the choice the negative root in: 
\[ z_k = \frac{1 - \sqrt{ (1-2\alpha)^2 + 4 \alpha f_k }}{2 \alpha} \le 1. \] Assembling these pieces yields \eqref{eq:newton-recur}.

\noindent \textbf{Convergence} We have an easy result that $f_{k+1} \le \frac{1}{4} f_k$ by ignoring the term $(1-2\alpha)^2$ in the denominator. Also, by direct evaluation, $f_1 = \alpha (1-\alpha)^2$. Thus, 
\[ f_{k} \le \frac{1}{4^{k-1}} f_1 = \frac{1}{4^{k-1}} \alpha (1-\alpha)^2, \] which is one side of the convergence rate. The sequence for $f_k$ also converges quadratically in the limit because $\lim_{k \to \infty} \frac{f_{k+1}}{f_{k}^2}  = \frac{\alpha}{(1-2 \alpha)^2}$.
\end{proof}

\paragraph{A practical, always-stochastic Newton iteration}
The Newton iteration from Theorem~\ref{thm:newton} begins as $\vx\itn{0} = 0$ and, when $\alpha < 1/2$, gradually grows the solution until it becomes stochastic and attains optimality. For problems when $\alpha > 1/2$, however, this iteration often converges to a fixed point where $\vx$ is not stochastic. (In fact, it always did this in our brief investigations.) To make our codes practical for problems where $\alpha > 1/2$, then, we enforce an explicit stochastic normalization after each Newton step:
\begin{equation} \label{eq:newton-stochastic}
[\eye - \alpha \mR(\vx\itn{k} \kron \mI + \mI \kron \vx\itn{k})] \vp\itn{k} = \vf(\vx\itn{k}) \quad \vx\itn{k+1} = \text{proj}(\vx\itn{k} + \vp\itn{k}) \quad \vx\itn{0} = (1-\alpha) \vv,
\end{equation}
where 
\[ \text{proj}(\vx) = \max(\vx,0) / \ve^\trans \max(\vx,0) \]
is a projection operator onto the probability simplex that sets negative elements to zero and then normalizes those left to sum to one.  We found this iteration superior to a few other choices including using a proximal point projection operator to produce always stochastic iterates~\cite[\S6.2.5]{Parikh-2014-prox}. We illustrate an example of the difference in Figure~\ref{fig:newton-stochastic} where the always-stochastic iteration solve the problem and the iteration without this projection converges to a non-stochastic fixed-point. Note that, like a general instance of Newton's method, the system $[\eye - \alpha \mR(\vx\itn{k} \kron \mI + \mI \kron \vx\itn{k})]$ may be singular. We never ran into such a case in our experiments and in our study.

\begin{figure}
\includegraphics[width=0.5\linewidth]{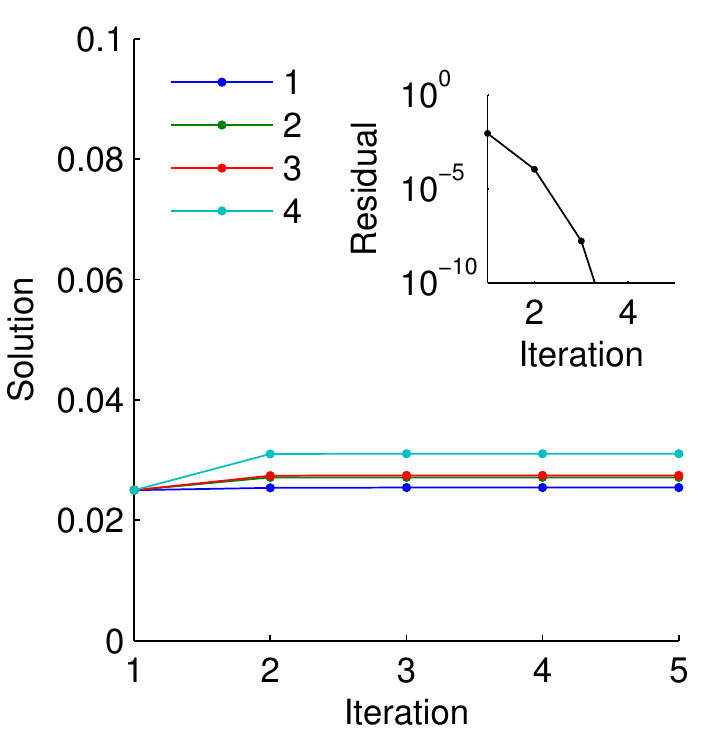}%
\includegraphics[width=0.5\linewidth]{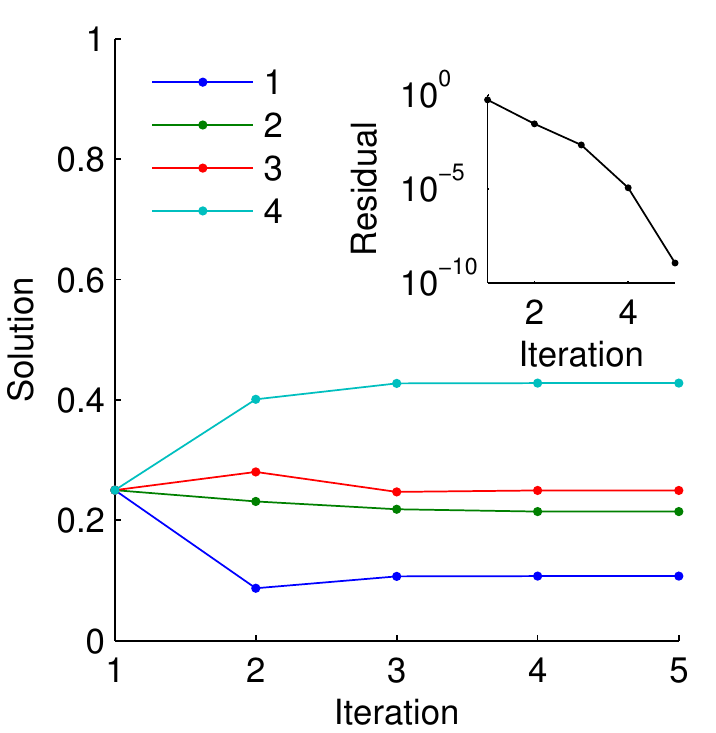}%

\caption{For problem $\mR_2$ with $\alpha = 0.99$, the Newton iteration from Theorem~\ref{thm:newton} (left figure) converges to a non-stochastic solution --- note the scale of the solution axis. The always-stochastic iteration~\eqref{eq:newton-stochastic} (right figure) converges for this problem. Thus, we recommend the iteration~\eqref{eq:newton-stochastic} when $\alpha > \frac{1}{m-1}$.}
\label{fig:newton-stochastic}
\end{figure}

We further illustrate the behavior of Newton's method on $\mR_2$ with $\alpha = 0.97$ and $\mR_2$ with $\alpha = 0.99$ in Figure~\ref{fig:newton}. The second of these problems did not converge for either the inner-outer or inverse iteration. Newton's method solves it in just a few iterations. In comparison to both the inner-outer and inverse iteration, however, Newton's method requires even more direct access to $\cmP$ in order to solve for the steps with the Jacobian.

\begin{figure}
\includegraphics[width=0.5\linewidth]{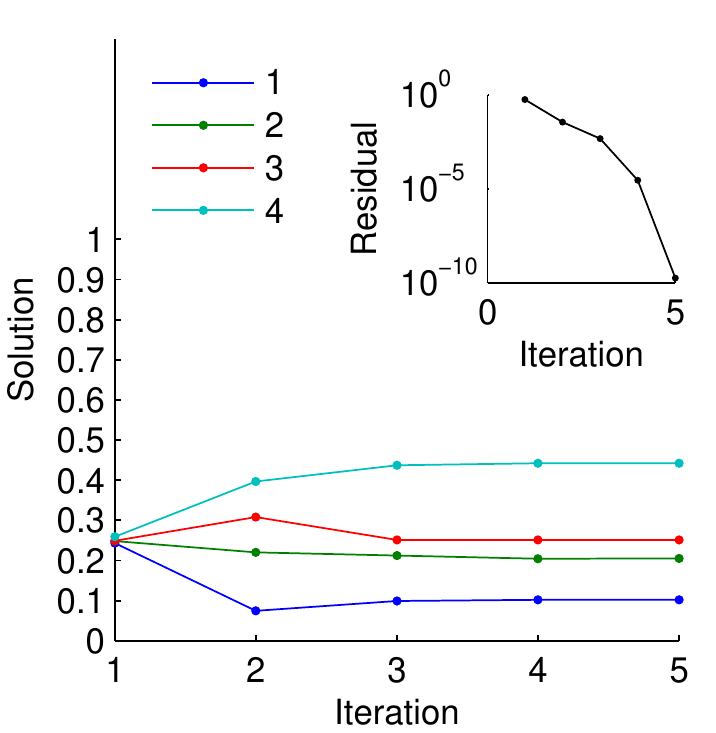}%
\includegraphics[width=0.5\linewidth]{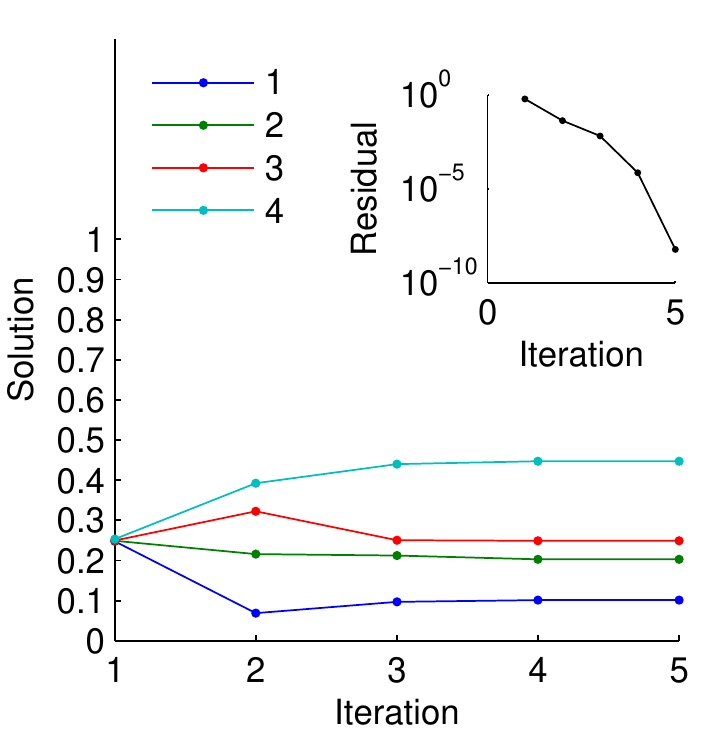}%

\caption{At left, the iterates of Newton's method to solve problem $\mR_2$ with $\alpha = 0.97$ and at right, the iterates to solve problem $\mR_2$ with $\alpha = 0.99$.  Both sequences converge unlike the inner-outer and inverse iterations.}
\label{fig:newton}
\end{figure}

\section{Experimental Results}
\label{sec:experiments}

To evaluate these algorithms, we create a database of problematic tensors. We then use this database to address two questions. 
\begin{enumerate}
\item What value of the shift is most reliable?
\item Which method has the most reliable convergence?
\end{enumerate}
In terms of reliability, we wish for the method to generate a residual smaller than $10^{-8}$ before reaching the maximum iteration count where the residual for a potential solution is given by~\eqref{eq:residual}. In many of the experiments, we run the methods for between 10,000 to 100,000 iterations. If we do not see convergence in this period, we deem a particular trial a failure. The value of $\vv$ is always $\ve /n$ but $\alpha$ will vary between our trials.
Before describing these results, we begin by discussing how we created the test problems.

\subsection{Problems}
We used exhaustive enumeration to identify $2 \times 2 \times 2$ and $3 \times 3 \times 3$ binary tensors, which we then normalized to stochastic tensors, that exhibited convergence problems with the fixed point or shifted methods. We also randomly sampled many $4 \times 4 \times 4$ binary problems and saved those that showed slow or erratic convergence for these same algorithms. We used $\alpha = 0.99$ and $\vv = \ve/n$ for these studies. The $6 \times 6 \times 6$ problems were constructed randomly in an attempt to be adversarial. Tensors with strong ``directionality'' seemed to arise as interesting cases in much of our theoretical study (this is not presented here). By this, we mean, for instance, tensors where a single state has many incoming links. We created a random procedure that generates problems where this is true (the exact method is in the online supplement) and used this to generate $6 \times 6 \times 6$ problems. In total, we have the following problems: 
\[ \begin{aligned}
3 \times 3 \times 3 & \qquad \text{5 problems} \\
4 \times 4 \times 4 & \qquad \text{19 problems} \\
6 \times 6 \times 6 & \qquad \text{5 problems}.
\end{aligned} \]
The full list of problems is in given in Appendix~\ref{app:problems}.

We used Matlab's symbolic toolbox to compute a set of exact solutions to these problems. These $6 \times 6 \times 6$ problems often had multiple solutions whereas the smaller problems only had a single solution (for the values of $\alpha$ we considered). While it is possible there are solutions missed by this tool, prior research found symbolic computation a reliable means of solving these polynomial systems of equations~\cite{Kolda-2011-sshopm}.

\subsection{Shifted iteration}
\label{sec:shiftstudy}
We begin our study by looking at a problem where the necessary shift suggested by the ODE theory ($\gamma = 1/2$ for third-order data) does not result in convergence. We are interested in whether or not varying the shift will alter the convergence behavior. This is indeed the case. For the problem $\mR_{4,11}$ from the appendix with $\alpha = 0.99$, we show the convergence of the residual as the shift $\gamma$ varies in Figure~\ref{fig:shiftstudy}. When $\gamma = 0.5$, the iteration does not converge. There is a point somewhere between $\gamma = 0.554$ and $\gamma = 0.5545$ where the iteration begins to converge. When we set $\gamma = 1$, the iteration converged rapidly. 

In the next experiment, we wished to understand how the reliability of the method depended on the shift $\gamma$. In Table~\ref{tab:shifttable}, we vary $\alpha$ and the shift $\gamma$ and look at how many of the $29$ test problems the shifted method can solve within 10,000 iterations. Recall that a method solves a problem if it pushes the residual below $10^{-8}$ within the iteration bound. The results from that table show that $\gamma = 1$ or $\gamma = 2$ results in the most reliable method. When $\gamma = 10$, then the method was less reliable. This is likely due to the shift delaying convergence for too long. Note that we chose many of the problems based on the failure of the shifted method with $\gamma = 0$ or $\gamma = 1/2$ and so the poor performance of these choices may not reflect their true reliability. Nevertheless, based on the results of this table, we recommend a shift of $\gamma = 1$ for a third-order problem, or a shift of $\gamma = m-2$ for a problem with an order-$m$ tensor. 

\begin{figure}
\centering
 \includegraphics{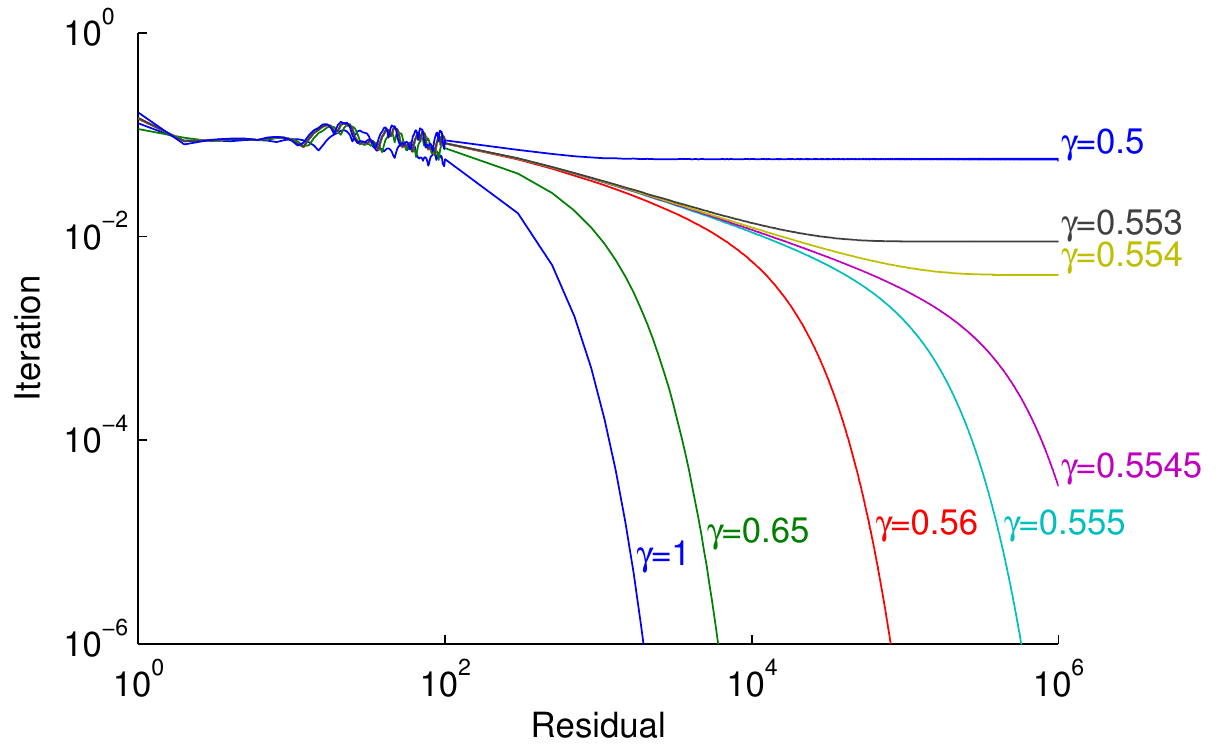}
 \caption{For the problem $\mR_{4,11}$ with $\alpha = 0.99$ and $\vv = \ve/n$, the shifted method will not converge unless $\gamma$ is slightly larger than 0.554. As $\gamma$ becomes larger, the convergence rate increases.}
 \label{fig:shiftstudy}
\end{figure}

\begin{table}
\caption{Each row of this table reports the number of problems successfully solved by the shifted iteration as the shift varies from $0$ to $10$. The values are reported for each value of $\alpha$ considered, as well as broken down into the different problem sizes considered.   }
\label{tab:shifttable}
 \centering
 \begin{tabularx}{0.7\linewidth}{>{\hsize=1.35\hsize}X >{\hsize=1.35\hsize}X *{7}{>{\hsize=0.9\hsize}X}}
  \toprule
$\alpha$ & $n$ & \multicolumn{7}{l}{Shifts $\gamma$}  \\
&     & 0 & 1/4 & 1/2 & 3/4 & 1 & 2 & 10 \\
\midrule
  0.70  &  3  & 5  & 5  & 5  & 5  & 5  & 5  & 5  \\ 
       &  4  & 19  & 19  & 19  & 19  & 19  & 19  & 19  \\ 
       &  6  & 5  & 5  & 5  & 5  & 5  & 5  & 5  \\ \addlinespace 
       &     & 29  & 29  & 29  & 29  & 29  & 29  & 29  \\ \midrule
 0.85  &  3  & 5  & 5  & 5  & 5  & 5  & 5  & 5  \\ 
       &  4  & 19  & 19  & 19  & 19  & 19  & 19  & 19  \\ 
       &  6  & 5  & 5  & 5  & 5  & 5  & 5  & 5  \\ \addlinespace 
       &     & 29  & 29  & 29  & 29  & 29  & 29  & 29  \\ \midrule
 0.90  &  3  & 5  & 5  & 5  & 5  & 5  & 5  & 5  \\ 
       &  4  & 18  & 19  & 19  & 19  & 19  & 19  & 19  \\ 
       &  6  & 5  & 5  & 5  & 5  & 5  & 5  & 5  \\ \addlinespace 
       &     & 28  & 29  & 29  & 29  & 29  & 29  & 29  \\ \midrule
 0.95  &  3  & 5  & 5  & 5  & 5  & 5  & 5  & 5  \\ 
       &  4  & 7  & 11  & 13  & 13  & 16  & 19  & 18  \\ 
       &  6  & 5  & 5  & 5  & 5  & 5  & 5  & 5  \\ \addlinespace 
       &     & 17  & 21  & 23  & 23  & 26  & 29  & 28  \\ \midrule
 0.99  &  3  & 4  & 5  & 5  & 5  & 5  & 5  & 5  \\ 
       &  4  & 0  & 1  & 1  & 2  & 2  & 2  & 2  \\ 
       &  6  & 1  & 1  & 1  & 2  & 2  & 2  & 1  \\ \addlinespace 
       &     & 5  & 7  & 7  & 9  & 9  & 9  & 8  \\ 
\bottomrule
\end{tabularx}
\end{table}

\subsection{Solver reliability}
\label{sec:solvers-perf}

In our final study, we utilize each method with the following default parameters: 
\begin{center}
 \begin{tabular}{lll}
  F & fixed point & 10,000 maximum iterations, $\vx_0 = \vv$  \\
  S & shifted & 10,000 maximum iterations, $\gamma = 1$, $\vx_0 = \vv$  \\
  IO & inner-outer & 1,000 outer iterations, internal tolerance $\eps$, $\vx_0 = \vv$  \\
  Inv & inverse & 1,000 iterations, $\vx_0 = \vv$  \\
  N & Newton & 1,000 iterations, projection step, $\vx_0 = (1-\alpha) \vv$ 
 \end{tabular}
\end{center}
We also evaluate each method with $10$ times the default number of iterations. 

The results of the evaluation are shown in Table~\ref{tab:methods} as $\alpha$ varies from $0.7$ to $0.99$. The fixed point method has the worst performance when $\alpha$ is large. Curiously, when $\alpha = 0.99$ the shifted method outperforms the inverse iteration, but when $\alpha = 0.95$ the inverse iteration outperforms the shifted iteration. This implies that the behavior and reliability of the methods is not monotonic in $\alpha$. While this fact is not overly surprising, it is pleasing to see a concrete example that might suggest some tweaks to the methods to improve their reliability. Overall, the inner-outer and Newton's method have the most reliable convergence on these difficult problems.

\begin{table}
\caption{Each row of this table reports the number of problems successfully solved by the various iterative methods in two cases: with their default parameters, and with 10 times the standard number of iterations. The values are reported for each value of $\alpha$ considered, as well as broken down into the different problem sizes considered. The columns are: F for the fixed-point, S for the shifted method, IO for the inner-outer, Inv for the inverse iteration, and N for Newton's method. }
\label{tab:methods}
\centering
\begin{tabularx}{\linewidth}{>{\hsize=1.5\hsize}X >{\hsize=1.5\hsize}X *{5}{>{\hsize=0.9\hsize}X} @{\qquad\qquad} *{5}{>{\hsize=0.9\hsize}X} }
\toprule
$\alpha$ & $n$ & \multicolumn{5}{l}{Method (defaults)} & \multicolumn{5}{l}{Method (Extra iteration)} \\
         &     & F & S & IO & Inv & N  & F & S & IO & Inv & N\\
\midrule
 0.70  &  3  & 5  & 5  & 5  & 5  & 5  & 5  & 5  & 5  & 5  & 5  \\ 
       &  4  & 19  & 19  & 19  & 19  & 19  & 19  & 19  & 19  & 19  & 19  \\ 
       &  6  & 5  & 5  & 5  & 5  & 5  & 5  & 5  & 5  & 5  & 5  \\ \addlinespace 
       &     & 29  & 29  & 29  & 29  & 29  & 29  & 29  & 29  & 29  & 29  \\ \midrule 
 0.85  &  3  & 5  & 5  & 5  & 5  & 5  & 5  & 5  & 5  & 5  & 5  \\ 
       &  4  & 19  & 19  & 19  & 19  & 19  & 19  & 19  & 19  & 19  & 19  \\ 
       &  6  & 5  & 5  & 5  & 5  & 5  & 5  & 5  & 5  & 5  & 5  \\ \addlinespace 
       &     & 29  & 29  & 29  & 29  & 29  & 29  & 29  & 29  & 29  & 29  \\ \midrule 
 0.90  &  3  & 5  & 5  & 5  & 5  & 5  & 5  & 5  & 5  & 5  & 5  \\ 
       &  4  & 18  & 19  & 19  & 19  & 19  & 18  & 19  & 19  & 19  & 19  \\ 
       &  6  & 5  & 5  & 5  & 5  & 5  & 5  & 5  & 5  & 5  & 5  \\ \addlinespace 
       &     & 28  & 29  & 29  & 29  & 29  & 28  & 29  & 29  & 29  & 29  \\ \midrule 
 0.95  &  3  & 5  & 5  & 5  & 5  & 5  & 5  & 5  & 5  & 5  & 5  \\ 
       &  4  & 7  & 16  & 18  & 19  & 19  & 8  & 16  & 19  & 19  & 19  \\ 
       &  6  & 5  & 5  & 5  & 5  & 5  & 5  & 5  & 5  & 5  & 5  \\ \addlinespace 
       &     & 17  & 26  & 28  & 29  & 29  & 18  & 26  & 29  & 29  & 29  \\ \midrule 
 0.99  &  3  & 4  & 5  & 5  & 5  & 5  & 4  & 5  & 5  & 5  & 5  \\ 
       &  4  & 0  & 2  & 15  & 1  & 19  & 0  & 2  & 17  & 1  & 19  \\ 
       &  6  & 1  & 2  & 3  & 1  & 4  & 2  & 3  & 4  & 3  & 4  \\ \addlinespace 
       &     & 5  & 9  & 23  & 7  & 28  & 6  & 10  & 26  & 9  & 28  \\  
\bottomrule			
\end{tabularx}
\end{table}

Newton's method, in fact, solves all but one instance: $\mR_{6,3}$ with $\alpha = 0.99$. We explore this problem in slightly more depth in Figure~\ref{fig:newton-failure}. This problem only has a single unique solution (based on our symbolic computation). However, none of the iterations will find it using the default settings --- all the methods are attracted to a point with a small residual and an indefinite Jacobian. We were able to find the true solution by using Newton's method with random starting points. It seems that the iterates need to approach the solution from on a rather precise trajectory in order to overcome an indefinite region. This problem should be a useful case for future algorithmic studies on the problem. 

\begin{figure}
\centering
 \includegraphics{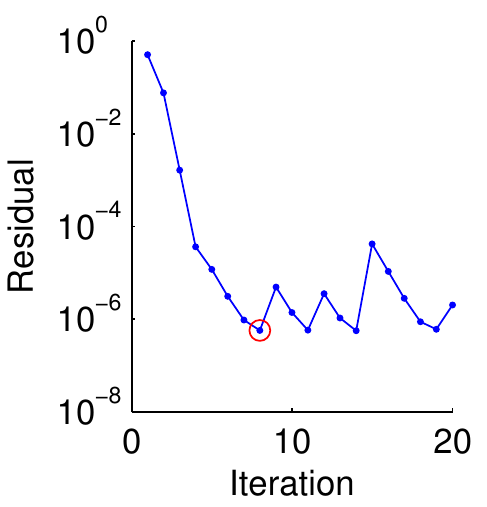}%
\begin{minipage}[b]{0.6\linewidth}

\normalsize The point ~~~~~~~~~ The eigenvalues ~~ The true solution

\scriptsize%
\begin{verbatim}
 0.199907259533067   0.980000000000000   0.043820721946272
 0.006619352098700   0.000064771773360   0.002224192630620
 0.116429656827957  -1.786544142144891   0.009256490884022
 0.223220491129316  -0.575965838505486   0.819168263512464
 0.079958855790239  -0.575965838505486   0.031217440669761
 0.373864384620721  -1.438690261635567   0.094312890356862
\end{verbatim}
\bigskip

\normalsize The Jacobian

\scriptsize
\begin{verbatim}
 -0.9712    0.2246    0.3496    0.1944    0.3395    0.7435
       0   -0.7299    0.0131         0    0.0824         0
  0.4781    0.1851   -0.9505         0    0.4621    0.2408
  0.0288    0.1851    0.0495    0.1822         0    0.4453
       0    0.4192    0.3701    0.0857   -0.5939    0.1581
  1.4443    0.6960    1.1482    0.5176    0.6899   -0.6077
\end{verbatim}\bigskip

\end{minipage}%

\caption{Newton's method on the non-convergent case of $\mR_{6,3}$ with $\alpha = 0.99$. The method repeatedly drops to a small residual before moving away from the solution. This happens because when the residual becomes small, then the Jacobian acquires a non-trivial positive eigenvalue as exemplified by the point with the red-circled residual. We show the Jacobian and eigenvalues at this point. It appears to be a pseudo-solution that attracts all of the algorithms. Using random starting points, Newton's method will sometimes generate the true solution, which is far from the the attracting point.}
\label{fig:newton-failure}
\end{figure}

\section{Discussion}

In this manuscript, we studied the higher-order PageRank problem as well as the multilinear PageRank problem. The higher-order PageRank problem behaves much like the standard PageRank problem: we always have guaranteed uniqueness and fast convergence. The multilinear PageRank problem, in contrast, only has uniqueness and fast convergence in a more narrow regime. Outside of that regime, existence of a solution is guaranteed, although uniqueness is not. As we were finalizing our manuscript for submission, we discovered an independent preprint that discusses some related results from an eigenvalue perspective~\cite{Chu-2014-eigenvalue}. 

For the multilinear PageRank problem, convergence of an iterative method outside of the uniqueness regime is highly dependent on the data. We created a test set based on problems where both the fixed-point and shifted fixed-point method fails. On these tough problems, both the inner-outer and Newton iterations had the best performance. This result suggests a two-phase approach to solving the problems: first try the simple shifted method. If that does not seem to converge, then use either a Newton or inner-outer iteration. Our empirical findings are limited to the third order case and we plan to revisit such strategies in the future when we consider large scale implementations of these methods on real-world problems --- the present efforts are focused on understanding what is and is not possible with the multilinear PageRank problem. This is also due to the observation that the multilinear PageRank problem is only interesting for massive problems. If $O(n^2)$ memory is available, then the higher-order PageRank vector should be used instead, unless there is a modeling reason to choose the multilinear PageRank formulation. 

Based on our theoretical results, we note that there seems to be a key transition for \emph{all} of the algorithms and theory that arises at the uniqueness threshold: $\alpha < 1/(m-1)$. We are currently trying to find algorithms with guaranteed convergence when $\alpha > 1/(m-1)$ but have not been successful yet. We plan to explore using sum-of-squares programming for this task in the future. Such an approach has given one of the first algorithms with good guarantees for the tensor eigenvalue problem~\cite{Nie-2013-sdp}.
 
\section*{Acknowledgments}
We are grateful to Austin Benson for suggesting an idea that led to the stochastic process as well as some preliminary comments on the manuscript. DFG would like to acknowledge support for NSF CCF-1149756. LHL gratefully acknowledges support for AFOSR FA9550-13-1-0133, NSF DMS-1209136, and NSF DMS-1057064.

\bibliographystyle{dgleich-bib}
\bibliography{all-bibliography}

\appendix

\section{Applying Li and Ng's results to multilinear PageRank}
\label{sec:pagerank-beta}

For the third-order tensor problem 
\[ \cmP \vx^2 = \vx \]
\citet{Li-2013-tensor-markov-chain} define a quantity called $\beta$ to determine if the solution is unique. (Their quantity was $\gamma$, but we use $\beta$ here to avoid confusion with the shifting parameter $\gamma$.) When $\beta > 1$, then the solution is unique and in this section, we show that $\beta > 1$ is a stronger condition than $\alpha < 1/2$. 
The scalar value $\beta$ ($0 \le \beta \le 4$) is defined:
\[
 \bigbetasub
\]
where $\langle n \rangle \equiv \{1,2,\cdots,n\}, S \subset \langle n \rangle,$ and $\bar{S} = \langle n \rangle \setminus S.$ Note that we divide this up into two components, $\beta_1$ and $\beta_2$ that both depend on the set $S$. 

When we apply their theory to multilinear PageRank, we study the problem: 
\[ \cmPbar \vx^2 = \vx \text{ where } \cel{\bar{P}}_{ijk} = \alpha \cel{P}_{ijk} + (1-\alpha) v_i. \]
The value of $\beta$ is a function of $\tensor{P}$ and clearly $\beta(\omega \tensor{P}) = \omega \beta(\tensor{P})$, 
where $\omega$ is a scalar. Generally, $\beta(\tensor{P} + \tensor{Q}) \neq \beta(\tensor{P}) + \beta(\tensor{Q})$ for arbitrary tensors $\tensor{P}$ and $\tensor{Q}$. However, the equation $\beta(\tensor{P} + \tensor{Q}) = \beta(\tensor{P}) + \beta(\tensor{Q})$ holds for the construction of $\cmPbar$ as we now show.

Let $\cmQ$ be the tensor where $\cQ_{ijk} = (1-\alpha)v_i$ and define $q_i = (1-\alpha) v_i$ to simplify the notation. Then $\cmPbar = \alpha \cmP + \cmQ$. Let us first consider $\beta_1$: 
\[ \begin{aligned}
    \beta_1(\alpha \cmP + \cmQ) &= \min_{k \in \langle n \rangle}\bigg[ \min_{j \in S} \sum_{i \in \bar{S}}\Big(\alpha \elm{P}_{ijk} + \elm{Q}_{ijk}\Big) + \min_{j \in \bar{S}} \sum_{i \in S}\Big(\alpha \elm{P}_{ijk} + \elm{Q}_{ijk}\Big) \bigg]\\
    &= \min_{k \in \langle n \rangle}\bigg[ \min_{j \in S} \sum_{i \in \bar{S}}\Big(\alpha \elm{P}_{ijk} + q_i\Big) + \min_{j \in \bar{S}} \sum_{i \in S}\Big(\alpha \elm{P}_{ijk} + q_{i}\Big) \bigg] \\
&= \min_{k \in \langle n \rangle}\bigg[ \Bigl( \min_{j \in S} \sum_{i \in \bar{S}} \alpha \elm{P}_{ijk} + \min_{j \in \bar{S}} \sum_{i \in S} \alpha \elm{P}_{ijk} \Bigr)  + \sum_i q_i\bigg] \\
&= \beta_1(\alpha \cmP) + \beta_1(\cmQ).
   \end{aligned} \]
By the same derivation, 
\[ \beta_2(\alpha \cmP + \cmQ) = \beta_2(\alpha \cmP) + \beta_2(\cmQ). \]
Now note that, because $\beta_1(\cmQ) = \beta_2(\cmQ) = 1-\alpha$ independently of the set $S$, we have 
\[ \beta(\alpha \cmP + \cmQ) = \alpha  \beta(\cmP) + 2(1-\alpha). \]

We are interested in the case that $\beta (\alpha \cmP + \cmQ) > 1$ to apply the uniqueness theorem.
Note that $\beta > 1$ can be true even if $\alpha > 1/2$. However, $\alpha < 1/2$ implies that $\beta > 1$.
Thus, the condition $\beta > 1$ is stronger.

\section{The tensor set} \label{app:problems}
The following problems gave us the tensors $\cmP$ for our experiments, after they were normalized to be column stochastic matrices.
\allowdisplaybreaks\setlength{\arraycolsep}{2pt}
\subsection{3 $\times$ 3 $\times$ 3}
\footnotesize
\begin{align*} 
\mR_{3,1} & = \left[ \scriptstyle\begin{array}{ccc|ccc|ccc}
1 & 1 & 1 & 1 & 0 & 0 & 0 & 0 & 0 \\
1 & 1 & 1 & 1 & 0 & 1 & 1 & 0 & 1 \\
1 & 1 & 1 & 1 & 1 & 1 & 0 & 1 & 0 \\
\end{array}
\right] \\ 
\mR_{3,2} & = \left[ \scriptstyle\begin{array}{ccc|ccc|ccc}
0 & 0 & 0 & 1 & 0 & 1 & 1 & 1 & 0 \\
0 & 0 & 0 & 0 & 1 & 1 & 0 & 0 & 0 \\
1 & 1 & 1 & 0 & 1 & 0 & 1 & 0 & 1 \\
\end{array}
\right] \\ 
\mR_{3,3} & = \left[ \scriptstyle\begin{array}{ccc|ccc|ccc}
0 & 1 & 0 & 1 & 0 & 1 & 1 & 1 & 0 \\
0 & 0 & 0 & 0 & 1 & 0 & 0 & 1 & 0 \\
1 & 1 & 1 & 1 & 1 & 0 & 1 & 0 & 1 \\
\end{array}
\right] \\ 
\mR_{3,4} & = \left[ \scriptstyle\begin{array}{ccc|ccc|ccc}
0 & 0 & 1 & 1 & 0 & 0 & 1 & 1 & 1 \\
0 & 0 & 1 & 1 & 0 & 0 & 0 & 0 & 1 \\
1 & 1 & 1 & 1 & 1 & 1 & 1 & 1 & 0 \\
\end{array}
\right] \\ 
\mR_{3,5} & = \left[ \scriptstyle\begin{array}{ccc|ccc|ccc}
0 & 0 & 0 & 0 & 0 & 0 & 1 & 0 & 1 \\
0 & 0 & 0 & 0 & 1 & 1 & 0 & 1 & 0 \\
1 & 1 & 1 & 1 & 0 & 0 & 0 & 0 & 0 \\
\end{array}
\right] \\ 
\end{align*} 

\subsection{4 $\times$ 4 $\times$ 4}
\footnotesize
\begin{align*} 
\mR_{4,1} & = \left[ \scriptstyle\begin{array}{cccc|cccc|cccc|cccc}
0 & 0 & 0 & 0 & 0 & 0 & 0 & 0 & 0 & 0 & 0 & 0 & 1 & 0 & 0 & 1 \\
0 & 0 & 0 & 0 & 0 & 1 & 0 & 1 & 0 & 1 & 0 & 0 & 0 & 1 & 0 & 0 \\
0 & 0 & 0 & 0 & 0 & 0 & 1 & 0 & 0 & 1 & 1 & 0 & 0 & 0 & 0 & 0 \\
1 & 1 & 1 & 1 & 1 & 0 & 0 & 0 & 1 & 0 & 0 & 1 & 1 & 0 & 1 & 0 \\
\end{array}
\right] \\ 
\mR_{4,2} & = \left[ \scriptstyle\begin{array}{cccc|cccc|cccc|cccc}
0 & 0 & 0 & 0 & 0 & 0 & 0 & 0 & 0 & 0 & 0 & 0 & 0 & 1 & 0 & 0 \\
0 & 0 & 0 & 0 & 1 & 1 & 1 & 0 & 0 & 1 & 0 & 0 & 0 & 1 & 0 & 0 \\
0 & 0 & 0 & 0 & 0 & 0 & 0 & 0 & 1 & 0 & 1 & 1 & 0 & 0 & 1 & 0 \\
1 & 1 & 1 & 1 & 0 & 0 & 0 & 1 & 0 & 0 & 0 & 0 & 1 & 0 & 0 & 1 \\
\end{array}
\right] \\ 
\mR_{4,3} & = \left[ \scriptstyle\begin{array}{cccc|cccc|cccc|cccc}
0 & 0 & 0 & 0 & 0 & 0 & 0 & 0 & 0 & 1 & 0 & 0 & 1 & 0 & 0 & 1 \\
0 & 0 & 0 & 0 & 1 & 1 & 1 & 0 & 0 & 1 & 0 & 0 & 0 & 1 & 0 & 0 \\
0 & 0 & 0 & 0 & 0 & 0 & 0 & 0 & 0 & 0 & 1 & 1 & 0 & 0 & 1 & 0 \\
1 & 1 & 1 & 1 & 0 & 0 & 0 & 1 & 1 & 0 & 0 & 0 & 0 & 1 & 0 & 0 \\
\end{array}
\right] \\ 
\mR_{4,4} & = \left[ \scriptstyle\begin{array}{cccc|cccc|cccc|cccc}
0 & 0 & 0 & 0 & 0 & 0 & 0 & 0 & 0 & 0 & 0 & 0 & 0 & 0 & 1 & 1 \\
0 & 0 & 0 & 0 & 0 & 1 & 0 & 1 & 0 & 0 & 0 & 0 & 0 & 1 & 0 & 0 \\
0 & 0 & 0 & 0 & 0 & 0 & 1 & 0 & 0 & 1 & 1 & 0 & 0 & 0 & 1 & 0 \\
1 & 1 & 1 & 1 & 1 & 0 & 0 & 0 & 1 & 0 & 0 & 1 & 1 & 0 & 0 & 0 \\
\end{array}
\right] \\ 
\mR_{4,5} & = \left[ \scriptstyle\begin{array}{cccc|cccc|cccc|cccc}
0 & 0 & 0 & 0 & 0 & 0 & 0 & 0 & 1 & 0 & 0 & 0 & 1 & 0 & 0 & 1 \\
0 & 0 & 0 & 0 & 0 & 1 & 0 & 1 & 0 & 0 & 0 & 0 & 0 & 1 & 0 & 0 \\
0 & 0 & 0 & 0 & 0 & 0 & 1 & 0 & 0 & 1 & 1 & 0 & 0 & 0 & 0 & 0 \\
1 & 1 & 1 & 1 & 1 & 0 & 0 & 0 & 0 & 0 & 0 & 1 & 0 & 0 & 1 & 1 \\
\end{array}
\right] \\ 
\mR_{4,6} & = \left[ \scriptstyle\begin{array}{cccc|cccc|cccc|cccc}
0 & 0 & 0 & 0 & 0 & 0 & 0 & 0 & 0 & 0 & 0 & 1 & 1 & 0 & 0 & 0 \\
0 & 0 & 0 & 0 & 0 & 1 & 0 & 1 & 0 & 0 & 0 & 0 & 0 & 1 & 0 & 0 \\
0 & 0 & 0 & 0 & 0 & 0 & 1 & 0 & 0 & 1 & 1 & 0 & 0 & 0 & 1 & 0 \\
1 & 1 & 1 & 1 & 1 & 0 & 0 & 0 & 1 & 0 & 0 & 0 & 0 & 0 & 1 & 1 \\
\end{array}
\right] \\ 
\mR_{4,7} & = \left[ \scriptstyle\begin{array}{cccc|cccc|cccc|cccc}
0 & 0 & 0 & 0 & 0 & 0 & 0 & 0 & 0 & 0 & 0 & 0 & 1 & 0 & 1 & 1 \\
0 & 0 & 0 & 0 & 0 & 1 & 0 & 1 & 0 & 1 & 0 & 0 & 1 & 1 & 0 & 0 \\
0 & 0 & 0 & 0 & 0 & 0 & 1 & 0 & 0 & 1 & 1 & 0 & 0 & 0 & 0 & 0 \\
1 & 1 & 1 & 1 & 1 & 0 & 0 & 0 & 1 & 0 & 0 & 1 & 0 & 0 & 1 & 0 \\
\end{array}
\right] \\ 
\mR_{4,8} & = \left[ \scriptstyle\begin{array}{cccc|cccc|cccc|cccc}
0 & 0 & 0 & 0 & 0 & 0 & 0 & 0 & 0 & 0 & 0 & 1 & 0 & 0 & 1 & 0 \\
0 & 0 & 0 & 0 & 0 & 1 & 0 & 1 & 0 & 0 & 0 & 0 & 0 & 1 & 0 & 0 \\
0 & 0 & 0 & 0 & 0 & 0 & 1 & 0 & 0 & 1 & 1 & 0 & 0 & 0 & 0 & 0 \\
1 & 1 & 1 & 1 & 1 & 0 & 0 & 0 & 1 & 0 & 0 & 0 & 1 & 0 & 0 & 1 \\
\end{array}
\right] \\ 
\mR_{4,9} & = \left[ \scriptstyle\begin{array}{cccc|cccc|cccc|cccc}
0 & 0 & 0 & 0 & 0 & 0 & 0 & 0 & 0 & 0 & 0 & 1 & 1 & 0 & 1 & 0 \\
0 & 0 & 0 & 0 & 0 & 1 & 0 & 1 & 0 & 1 & 0 & 0 & 0 & 1 & 0 & 0 \\
0 & 0 & 0 & 0 & 0 & 0 & 1 & 0 & 0 & 1 & 1 & 0 & 0 & 0 & 0 & 0 \\
1 & 1 & 1 & 1 & 1 & 0 & 0 & 0 & 1 & 0 & 0 & 0 & 0 & 0 & 0 & 1 \\
\end{array}
\right] \\ 
\mR_{4,10} & = \left[ \scriptstyle\begin{array}{cccc|cccc|cccc|cccc}
0 & 0 & 0 & 0 & 0 & 0 & 0 & 1 & 0 & 0 & 0 & 0 & 1 & 0 & 0 & 1 \\
0 & 0 & 0 & 0 & 0 & 1 & 1 & 0 & 0 & 1 & 0 & 0 & 0 & 1 & 0 & 0 \\
0 & 0 & 0 & 0 & 0 & 0 & 0 & 0 & 0 & 0 & 1 & 1 & 0 & 1 & 1 & 0 \\
1 & 1 & 1 & 1 & 1 & 0 & 0 & 0 & 1 & 0 & 0 & 0 & 0 & 0 & 0 & 0 \\
\end{array}
\right] \\ 
\mR_{4,11} & = \left[ \scriptstyle\begin{array}{cccc|cccc|cccc|cccc}
0 & 0 & 0 & 0 & 0 & 0 & 0 & 0 & 0 & 0 & 0 & 0 & 0 & 0 & 0 & 1 \\
0 & 0 & 0 & 0 & 0 & 1 & 0 & 1 & 0 & 1 & 0 & 0 & 0 & 1 & 0 & 0 \\
0 & 0 & 0 & 0 & 0 & 0 & 1 & 0 & 0 & 1 & 1 & 1 & 0 & 0 & 0 & 0 \\
1 & 1 & 1 & 1 & 1 & 0 & 0 & 0 & 1 & 0 & 0 & 0 & 1 & 1 & 1 & 1 \\
\end{array}
\right] \\ 
\mR_{4,12} & = \left[ \scriptstyle\begin{array}{cccc|cccc|cccc|cccc}
0 & 0 & 0 & 0 & 0 & 0 & 0 & 1 & 0 & 0 & 0 & 0 & 1 & 0 & 0 & 1 \\
0 & 0 & 0 & 0 & 0 & 1 & 1 & 0 & 0 & 1 & 0 & 0 & 0 & 0 & 0 & 0 \\
0 & 0 & 0 & 0 & 0 & 0 & 0 & 0 & 0 & 0 & 1 & 1 & 0 & 0 & 1 & 0 \\
1 & 1 & 1 & 1 & 1 & 0 & 0 & 0 & 1 & 0 & 0 & 1 & 1 & 1 & 0 & 1 \\
\end{array}
\right] \\ 
\mR_{4,13} & = \left[ \scriptstyle\begin{array}{cccc|cccc|cccc|cccc}
0 & 0 & 0 & 0 & 0 & 0 & 0 & 0 & 0 & 0 & 0 & 0 & 0 & 0 & 1 & 1 \\
0 & 0 & 0 & 0 & 0 & 1 & 0 & 1 & 0 & 1 & 0 & 0 & 0 & 1 & 0 & 0 \\
0 & 0 & 0 & 0 & 0 & 0 & 1 & 0 & 1 & 1 & 1 & 0 & 0 & 0 & 0 & 0 \\
1 & 1 & 1 & 1 & 1 & 0 & 0 & 0 & 0 & 0 & 0 & 1 & 1 & 0 & 1 & 0 \\
\end{array}
\right] \\ 
\mR_{4,14} & = \left[ \scriptstyle\begin{array}{cccc|cccc|cccc|cccc}
0 & 0 & 0 & 0 & 0 & 0 & 0 & 0 & 1 & 0 & 0 & 0 & 0 & 1 & 0 & 0 \\
0 & 0 & 0 & 0 & 1 & 1 & 1 & 0 & 0 & 1 & 0 & 0 & 0 & 0 & 0 & 0 \\
0 & 0 & 0 & 0 & 0 & 0 & 0 & 0 & 0 & 0 & 1 & 1 & 0 & 0 & 1 & 0 \\
1 & 1 & 1 & 1 & 0 & 0 & 0 & 1 & 0 & 0 & 0 & 0 & 1 & 0 & 1 & 1 \\
\end{array}
\right] \\ 
\mR_{4,15} & = \left[ \scriptstyle\begin{array}{cccc|cccc|cccc|cccc}
0 & 0 & 0 & 0 & 0 & 0 & 0 & 0 & 0 & 0 & 0 & 0 & 0 & 0 & 1 & 0 \\
0 & 0 & 0 & 0 & 0 & 1 & 0 & 1 & 0 & 0 & 0 & 0 & 0 & 1 & 0 & 0 \\
0 & 0 & 0 & 0 & 0 & 0 & 1 & 0 & 1 & 1 & 1 & 0 & 0 & 0 & 1 & 0 \\
1 & 1 & 1 & 1 & 1 & 0 & 0 & 0 & 0 & 0 & 0 & 1 & 1 & 0 & 0 & 1 \\
\end{array}
\right] \\ 
\mR_{4,16} & = \left[ \scriptstyle\begin{array}{cccc|cccc|cccc|cccc}
0 & 0 & 0 & 0 & 0 & 0 & 0 & 0 & 0 & 0 & 0 & 0 & 1 & 0 & 0 & 1 \\
0 & 0 & 0 & 0 & 0 & 1 & 1 & 0 & 0 & 1 & 0 & 0 & 0 & 1 & 0 & 0 \\
0 & 0 & 0 & 0 & 0 & 0 & 0 & 0 & 0 & 1 & 1 & 1 & 1 & 0 & 1 & 0 \\
1 & 1 & 1 & 1 & 1 & 0 & 0 & 1 & 1 & 0 & 0 & 0 & 0 & 1 & 0 & 0 \\
\end{array}
\right] \\ 
\mR_{4,17} & = \left[ \scriptstyle\begin{array}{cccc|cccc|cccc|cccc}
0 & 0 & 0 & 0 & 1 & 0 & 0 & 0 & 0 & 0 & 0 & 0 & 0 & 1 & 0 & 1 \\
0 & 0 & 0 & 0 & 0 & 1 & 1 & 0 & 0 & 1 & 0 & 0 & 0 & 0 & 0 & 0 \\
0 & 0 & 0 & 0 & 0 & 0 & 0 & 0 & 0 & 1 & 1 & 1 & 0 & 1 & 1 & 0 \\
1 & 1 & 1 & 1 & 0 & 0 & 0 & 1 & 1 & 0 & 0 & 0 & 1 & 1 & 1 & 0 \\
\end{array}
\right] \\ 
\mR_{4,18} & = \left[ \scriptstyle\begin{array}{cccc|cccc|cccc|cccc}
0 & 0 & 0 & 0 & 1 & 0 & 0 & 0 & 0 & 0 & 0 & 0 & 1 & 1 & 0 & 1 \\
0 & 0 & 0 & 0 & 0 & 1 & 1 & 0 & 0 & 1 & 0 & 0 & 0 & 0 & 0 & 0 \\
0 & 0 & 0 & 0 & 0 & 0 & 0 & 0 & 0 & 0 & 1 & 1 & 0 & 0 & 1 & 0 \\
1 & 1 & 1 & 1 & 0 & 0 & 0 & 1 & 1 & 1 & 0 & 0 & 0 & 0 & 0 & 0 \\
\end{array}
\right] \\ 
\mR_{4,19} & = \left[ \scriptstyle\begin{array}{cccc|cccc|cccc|cccc}
0 & 0 & 0 & 0 & 0 & 0 & 0 & 0 & 0 & 0 & 0 & 0 & 1 & 0 & 0 & 1 \\
0 & 0 & 0 & 0 & 0 & 1 & 0 & 1 & 1 & 0 & 0 & 0 & 1 & 1 & 0 & 0 \\
0 & 0 & 0 & 0 & 0 & 0 & 1 & 0 & 0 & 1 & 1 & 0 & 1 & 0 & 0 & 0 \\
1 & 1 & 1 & 1 & 1 & 0 & 0 & 0 & 1 & 0 & 0 & 1 & 1 & 0 & 1 & 0 \\
\end{array}
\right] \\ 
\end{align*} 

\subsection{6 $\times$ 6 $\times$ 6}
\footnotesize
\begin{align*} 
\mR_{6,1} & = \left[ \scriptstyle\begin{array}{cccccc|cccccc|cccccc|cccccc|cccccc|cccccc}
0 & 0 & 0 & 0 & 0 & 0 & 0 & 0 & 0 & 0 & 0 & 0 & 0 & 1 & 0 & 0 & 0 & 0 & 1 & 0 & 0 & 0 & 1 & 0 & 1 & 0 & 0 & 0 & 0 & 1 & 1 & 0 & 0 & 0 & 1 & 0 \\
0 & 0 & 0 & 0 & 0 & 0 & 0 & 0 & 0 & 0 & 0 & 0 & 0 & 0 & 0 & 0 & 0 & 0 & 0 & 1 & 0 & 0 & 0 & 0 & 1 & 0 & 0 & 0 & 0 & 0 & 1 & 0 & 0 & 0 & 0 & 0 \\
0 & 0 & 0 & 0 & 0 & 0 & 0 & 0 & 0 & 0 & 0 & 0 & 0 & 0 & 1 & 0 & 0 & 1 & 0 & 0 & 1 & 0 & 0 & 0 & 0 & 1 & 0 & 0 & 0 & 0 & 0 & 0 & 1 & 0 & 0 & 0 \\
0 & 0 & 0 & 0 & 0 & 0 & 0 & 1 & 1 & 0 & 1 & 0 & 1 & 0 & 0 & 0 & 0 & 0 & 0 & 1 & 0 & 1 & 0 & 1 & 0 & 1 & 0 & 0 & 0 & 0 & 0 & 0 & 0 & 1 & 1 & 0 \\
0 & 0 & 0 & 0 & 0 & 0 & 0 & 0 & 0 & 0 & 1 & 0 & 0 & 0 & 0 & 0 & 1 & 0 & 0 & 0 & 0 & 0 & 0 & 0 & 0 & 0 & 1 & 0 & 1 & 0 & 0 & 1 & 0 & 0 & 0 & 0 \\
1 & 1 & 1 & 1 & 1 & 1 & 1 & 0 & 0 & 1 & 0 & 1 & 0 & 0 & 0 & 1 & 0 & 1 & 0 & 1 & 0 & 0 & 0 & 0 & 0 & 1 & 0 & 1 & 0 & 0 & 0 & 1 & 0 & 0 & 0 & 1 \\
\end{array}
\right] \\ 
\mR_{6,2} & = \left[ \scriptstyle\begin{array}{cccccc|cccccc|cccccc|cccccc|cccccc|cccccc}
0 & 0 & 0 & 0 & 0 & 0 & 0 & 0 & 0 & 0 & 0 & 0 & 1 & 0 & 0 & 0 & 0 & 0 & 0 & 0 & 0 & 0 & 0 & 0 & 1 & 1 & 0 & 0 & 0 & 1 & 0 & 1 & 0 & 0 & 0 & 1 \\
0 & 0 & 0 & 0 & 0 & 0 & 0 & 0 & 0 & 0 & 0 & 0 & 0 & 0 & 0 & 1 & 0 & 0 & 0 & 0 & 0 & 0 & 0 & 0 & 0 & 0 & 1 & 0 & 0 & 1 & 1 & 0 & 0 & 0 & 0 & 0 \\
0 & 0 & 0 & 0 & 0 & 0 & 0 & 1 & 0 & 0 & 0 & 1 & 0 & 0 & 0 & 0 & 0 & 0 & 0 & 0 & 0 & 0 & 0 & 0 & 0 & 0 & 1 & 0 & 0 & 0 & 1 & 0 & 0 & 0 & 1 & 0 \\
0 & 0 & 0 & 0 & 0 & 0 & 1 & 0 & 0 & 0 & 0 & 0 & 1 & 0 & 0 & 0 & 0 & 0 & 0 & 0 & 0 & 1 & 0 & 1 & 0 & 0 & 0 & 1 & 0 & 0 & 0 & 1 & 0 & 1 & 1 & 0 \\
0 & 0 & 0 & 0 & 0 & 0 & 0 & 0 & 0 & 0 & 0 & 0 & 0 & 0 & 1 & 0 & 0 & 0 & 0 & 0 & 0 & 0 & 1 & 0 & 0 & 0 & 0 & 1 & 1 & 0 & 0 & 0 & 0 & 0 & 0 & 0 \\
1 & 1 & 1 & 1 & 1 & 1 & 0 & 1 & 1 & 1 & 1 & 0 & 0 & 1 & 1 & 0 & 1 & 1 & 1 & 1 & 1 & 0 & 0 & 0 & 0 & 0 & 0 & 0 & 0 & 0 & 0 & 0 & 1 & 0 & 0 & 0 \\
\end{array}
\right] \\ 
\mR_{6,3} & = \left[ \scriptstyle\begin{array}{cccccc|cccccc|cccccc|cccccc|cccccc|cccccc}
0 & 0 & 0 & 0 & 0 & 0 & 0 & 0 & 0 & 0 & 0 & 1 & 1 & 0 & 0 & 0 & 0 & 0 & 0 & 0 & 1 & 0 & 1 & 0 & 0 & 1 & 1 & 0 & 0 & 0 & 0 & 0 & 0 & 0 & 0 & 1 \\
0 & 0 & 0 & 0 & 0 & 0 & 0 & 0 & 1 & 0 & 0 & 0 & 0 & 1 & 0 & 0 & 0 & 0 & 0 & 0 & 0 & 0 & 0 & 0 & 0 & 1 & 0 & 0 & 1 & 0 & 0 & 0 & 0 & 0 & 0 & 0 \\
0 & 0 & 0 & 0 & 0 & 0 & 0 & 0 & 0 & 0 & 0 & 0 & 1 & 0 & 0 & 0 & 0 & 0 & 0 & 0 & 0 & 0 & 0 & 0 & 1 & 0 & 0 & 0 & 1 & 0 & 1 & 1 & 0 & 0 & 1 & 0 \\
0 & 0 & 0 & 0 & 0 & 0 & 0 & 0 & 0 & 0 & 0 & 0 & 1 & 0 & 0 & 0 & 0 & 0 & 0 & 0 & 0 & 1 & 0 & 1 & 0 & 0 & 0 & 0 & 0 & 0 & 0 & 1 & 0 & 1 & 0 & 0 \\
0 & 0 & 0 & 0 & 0 & 0 & 0 & 1 & 0 & 0 & 0 & 1 & 0 & 0 & 0 & 0 & 0 & 0 & 0 & 1 & 0 & 0 & 0 & 0 & 0 & 0 & 0 & 1 & 0 & 0 & 0 & 0 & 1 & 0 & 1 & 0 \\
1 & 1 & 1 & 1 & 1 & 1 & 1 & 0 & 0 & 1 & 1 & 0 & 1 & 0 & 1 & 1 & 1 & 1 & 1 & 0 & 0 & 0 & 0 & 0 & 0 & 0 & 0 & 0 & 0 & 1 & 0 & 0 & 0 & 0 & 0 & 0 \\
\end{array}
\right] \\ 
\mR_{6,4} & = \left[ \scriptstyle\begin{array}{cccccc|cccccc|cccccc|cccccc|cccccc|cccccc}
0 & 0 & 0 & 0 & 0 & 0 & 0 & 0 & 0 & 0 & 0 & 0 & 0 & 0 & 0 & 0 & 0 & 0 & 0 & 0 & 0 & 0 & 0 & 0 & 0 & 0 & 0 & 0 & 0 & 0 & 1 & 0 & 0 & 0 & 0 & 0 \\
0 & 0 & 0 & 0 & 0 & 0 & 0 & 0 & 0 & 0 & 0 & 0 & 1 & 0 & 0 & 0 & 1 & 0 & 1 & 0 & 0 & 0 & 0 & 0 & 1 & 1 & 1 & 0 & 0 & 0 & 0 & 1 & 1 & 0 & 1 & 0 \\
0 & 0 & 0 & 0 & 0 & 0 & 0 & 0 & 0 & 0 & 0 & 0 & 0 & 0 & 1 & 1 & 0 & 0 & 0 & 0 & 1 & 0 & 0 & 0 & 1 & 0 & 0 & 0 & 0 & 1 & 1 & 0 & 1 & 0 & 1 & 0 \\
0 & 0 & 0 & 0 & 0 & 0 & 0 & 0 & 1 & 0 & 1 & 0 & 1 & 0 & 0 & 0 & 1 & 0 & 1 & 0 & 1 & 1 & 1 & 1 & 0 & 0 & 0 & 1 & 1 & 1 & 0 & 0 & 0 & 1 & 0 & 0 \\
0 & 0 & 0 & 0 & 0 & 0 & 0 & 0 & 0 & 0 & 1 & 0 & 0 & 0 & 0 & 0 & 0 & 0 & 0 & 0 & 0 & 0 & 0 & 0 & 0 & 1 & 0 & 0 & 0 & 0 & 0 & 0 & 0 & 0 & 1 & 0 \\
1 & 1 & 1 & 1 & 1 & 1 & 1 & 1 & 0 & 1 & 0 & 1 & 0 & 1 & 0 & 0 & 0 & 1 & 0 & 1 & 0 & 0 & 0 & 0 & 0 & 0 & 0 & 0 & 0 & 1 & 0 & 0 & 0 & 1 & 0 & 1 \\
\end{array}
\right] \\ 
\mR_{6,5} & = \left[ \begin{array}{cccccc|cccccc|cccccc|cccccc|cccccc|cccccc}
0 & 0 & 0 & 0 & 0 & 0 & 0 & 0 & 0 & 0 & 0 & 0 & 0 & 0 & 0 & 0 & 0 & 1 & 0 & 0 & 1 & 0 & 0 & 0 & 0 & 0 & 0 & 0 & 0 & 0 & 0 & 0 & 0 & 0 & 1 & 0 \\
0 & 0 & 0 & 0 & 0 & 0 & 0 & 0 & 0 & 1 & 0 & 0 & 0 & 0 & 0 & 0 & 1 & 0 & 0 & 0 & 0 & 0 & 0 & 0 & 1 & 1 & 0 & 0 & 0 & 0 & 1 & 1 & 0 & 1 & 0 & 0 \\
0 & 0 & 0 & 0 & 0 & 0 & 0 & 0 & 0 & 0 & 0 & 0 & 1 & 0 & 1 & 0 & 0 & 1 & 0 & 0 & 0 & 0 & 0 & 0 & 1 & 1 & 0 & 0 & 0 & 0 & 0 & 0 & 1 & 0 & 0 & 0 \\
0 & 0 & 0 & 0 & 0 & 0 & 0 & 1 & 1 & 1 & 0 & 0 & 0 & 0 & 0 & 0 & 1 & 0 & 0 & 1 & 1 & 0 & 1 & 0 & 1 & 0 & 0 & 1 & 0 & 0 & 0 & 1 & 0 & 1 & 1 & 0 \\
0 & 0 & 0 & 0 & 0 & 0 & 0 & 0 & 0 & 0 & 0 & 0 & 0 & 0 & 0 & 0 & 0 & 0 & 0 & 0 & 1 & 0 & 1 & 0 & 0 & 1 & 0 & 0 & 1 & 1 & 1 & 0 & 0 & 0 & 0 & 0 \\
1 & 1 & 1 & 1 & 1 & 1 & 1 & 0 & 0 & 0 & 1 & 1 & 0 & 1 & 0 & 1 & 0 & 0 & 1 & 1 & 1 & 1 & 0 & 1 & 0 & 0 & 1 & 0 & 1 & 0 & 0 & 0 & 0 & 0 & 0 & 1 \\
\end{array}
\right] \\ 
\end{align*} 
 
 \end{document}